\documentclass[11pt]{amsart}
\usepackage{amssymb,amsmath,amsfonts,amscd,euscript}
\usepackage{tikz}
\usepackage{tikz-cd}
\usepackage[all]{xy}
\usepackage{color}

\usepackage[backref=page]{hyperref}
\usepackage{todonotes}

\usepackage{amsthm,enumitem,bm,pict2e,xcolor}
\usepackage{indentfirst}

\usepackage{algorithm}
\usepackage{algpseudocode}
\usepackage{float}

\newtheorem{theoremA}{Theorem}

\newtheorem{conjA}[theoremA]{Conjecture}

\newcommand{\nc}{\newcommand}

\numberwithin{equation}{section}
\newtheorem{thm}{Theorem}[section]
\newtheorem{prop}[thm]{Proposition}
\newtheorem{lem}[thm]{Lemma}
\newtheorem{cor}[thm]{Corollary}
\newtheorem{rem}[thm]{Remark}

\newtheorem{example}[thm]{Example}
\newtheorem{dfn}[thm]{Definition}

\newtheorem{conj}[thm]{Conjecture}

\nc{\gl}{\mathfrak{gl}}
\nc{\GL}{\mathfrak{GL}}
\nc{\g}{\mathfrak{g}}
\nc{\gh}{\widehat\g}
\nc{\hh}{{\fh}^a}
\nc{\fh}{\mathfrak{h}}
\nc{\la}{\lambda}
\nc{\al}{\alpha }
\nc{\be}{\beta }
\nc{\om}{\omega }

\nc{\ta}{\theta}
\nc{\veps}{\varepsilon}
\nc{\ch}{{\mathop {\rm ch}}}
\nc{\Tr}{{\mathop {\rm Tr}\,}}
\nc{\Id}{{\mathop {\rm Id}}}
\nc{\Aut}{{\mathop {\rm Aut}}}
\nc{\End}{{\mathop {\rm End}}}

\nc{\bra}{\langle}
\nc{\ket}{\rangle}
\nc{\bs}{{\bf s}}
\nc{\bp}{{\bf p}}
\newcommand{\bI}{{\bf I}}
\nc{\pa}{\partial}
\nc{\ld}{\ldots}
\nc{\cd}{\cdots}
\nc{\hk}{\hookrightarrow}
\nc{\T}{\otimes}
\nc{\Gr}{\mathrm{Gr}}
\nc{\aGr}{{\mathcal Gr}}
\nc{\aFl}{{\mathcal Fl}}
\nc{\ov}{\overline}

\nc{\msl}{\mathfrak{sl}}
\nc{\mgl}{\mathfrak{gl}}
\nc{\U}{\mathrm U}
\nc{\Res}{\mathrm{Res\ }}

\newcommand{\bC}{{\mathbb C}}

\newcommand{\bZ}{{\mathbb Z}}

\newcommand{\bP}{{\mathbb P}}
\newcommand{\bA}{{\mathbb A}}
\newcommand{\bK}{{\mathbb K}}
\newcommand{\bO}{{\mathbb O}}
\newcommand{\bX}{{\mathbb X}}

\newcommand{\bfD}{{\bf D}}
\newcommand{\bfL}{{\bf L}}

\newcommand{\fg}{{\mathfrak g}}
\newcommand{\fb}{{\mathfrak b}}

\newcommand{\fI}{{\mathfrak I}}
\newcommand{\fn}{{\mathfrak n}}
\newcommand{\fa}{{\mathfrak a}}
\newcommand{\eO}{\EuScript{O}}

\newcommand{\eL}{\EuScript{L}}

\newcommand{\bk}{{\bf k}}
\newcommand{\barla}{{\bar\lambda}}

\tikzcdset{scale cd/.style={every label/.append style={scale=#1},
		cells={nodes={scale=#1}}}}
\begin{document}

\title[Type A algebraic coherence conjecture]
{Type A algebraic coherence conjecture of Pappas and Rapoport}

\author{Evgeny Feigin}
\address{School of Mathematical Sciences, Tel Aviv University, Tel Aviv, 69978, Israel}
\email{evgfeig@gmail.com}

\makeatletter\let\@wraptoccontribs\wraptoccontribs\makeatother
\contrib[an appendix in collaboration with]{Andrey Karenskih}
\email{karenskih@mail.tau.ac.il}

\begin{abstract}
The Pappas--Rapoport coherence conjecture, proved by Zhu, states that the dimensions of spaces
of sections of certain line bundles coincide. 
The two sides of the equality correspond to line
bundles on spherical Schubert varieties in affine Grassmannians and to line bundles on 
unions of Schubert varieties in affine flag varieties. Algebraically, the claim can be reformulated 
as an equality between the
dimensions of certain Demazure modules and certain sums of Demazure modules. The goal of this
paper is to formulate an algebraic construction that provides an explicit link between the 
aforementioned Demazure modules. Our construction works only in type A, but it applies to a much 
wider class of representations than those arising in the geometric coherence conjecture. 
In the general case, one side of the conjectural equality involves affine Kostant--Kumar modules.   	
\end{abstract}

\maketitle

\section{Introduction}
The goal of this paper is to introduce an explicit algebraic counterpart of the geometric 
Pappas--Rapoport coherence conjecture in type A (Zhu's theorem) \cite{PR08,Zhu14}. 
The input of our construction is a collection
of cyclic representations of the current algebra, the elements of the collection are
labeled by the simple roots of the underlying affine Kac--Moody Lie algebra. The output is
a (non-cyclic) representation of the Iwahori algebra; the character of this representation 
coincides with that of the Cartan component of the tensor product of the initial 
representations.  If one starts with a very special collection
of spherical Demazure submodules of integrable irreducible representations, whose highest weights are
multiples of a fixed level-one weight, then the resulting Iwahori representation is a sum of Demazure
submodules in another integrable irreducible representation. The input and output in this case
are exactly the left- and right-hand sides of the coherence conjecture of Pappas and Rapoport.
Let us provide more details.

Throughout the paper, we work with the Lie algebras $\mgl_n$ and $\msl_n$, and their affine 
Kac--Moody algebras \cite{Kac90}. 
Let $\aGr$ be the affine Grassmannian for the adjoint group of  $\msl_n$; 
in particular, $\aGr$ is the disjoint 
union of $n$ components $\aGr_b$, where 
$b=0,\dots,n-1$ \cite{Kac90,Kum02}. The ind-varieties $\aGr_b$ are covered
by spherical Schubert varieties of the form $X_\la$, where $\la$ are partitions with at most $n$ parts.
Each $X_\la$ is acted upon by the current group $SL_n[[z]]$, which is a subgroup of the affine Kac--Moody
Lie group $\widehat{SL}_n$ acting on $\aGr_b$ (in fact, $\aGr_b$ are partial flag varieties for 
$\widehat{SL}_n$). 

Let $\eO(1)$ be the ample line bundle on  $\aGr_b$, generating the corresponding Picard group.
Then the restricted dual of the space of sections $H^0(\aGr_b,\eO(1))$ is identified with the integrable 
irreducible highest weight representation $L(\Lambda_b)$ of the affine Kac--Moody Lie algebra 
$\widehat\msl_n$ and one has an embedding $\aGr_b\subset \bP(L(\Lambda_b))$. We note that for any 
$k\in\bZ_{>0}$, the highest weight $\widehat\msl_n$ module $L(k\Lambda_b)$ is identified with
the dual space of sections of the line bundle $\eO(k)=\eO(1)^{\T k}$. 
Now let us restrict $\eO(k)$ to $X_\la$. Then $H^0(X_\la,\eO(k))^*$ admits an action of the current algebra
$\msl_n[z]$ and is isomorphic to the Demazure module $D_{k,\la}\subset L(k\Lambda_b)$, which is a cyclic
representation of $\msl_n[z]$ with a cyclic vector of weight $k\la$.

The affine Grassmannians $\aGr_b$ admit degeneration to the complete flag variety $\aFl$.
More precisely, there exists a family over $\bA^1$ whose general fiber is isomorphic 
to $\aGr_b$ and the special fiber is isomorphic to $\aFl$ (see \cite{Ga01,Go10,Zhu17,AB24}). The family is defined
for arbitrary groups, but in type $A$ it admits a very concrete description 
in terms of lattices \cite{Zhu17,AR}.
If one considers a subfamily whose general fiber is the Schubert variety $X_\la$, then the special
fiber is identified with a union of Schubert subvarieties $Y_w$ inside the affine flag variety $\aFl$
\cite{Zhu14,Go01,PRS13}.
The varieties $Y_w$ are labeled by the elements of the so-called admissible set $\mathcal{A}_\la$,
which is a cardinality $|W|$ subset of the extended affine Weyl group ($W=S_n$ is the finite Weyl group). 
We note that $Y_w$ are, in general, not invariant
with respect to the whole current algebra, but admit an action of the Iwahori subalgebra $\fI$.

The Picard group of the affine flag variety has a collection of generators $\eL_b$,
$b=0,\dots,n-1$, where $\eL_b$ is the pullback of $\eO(1)$ with respect to the natural projection
map $\aFl\to\aGr_b$. In particular, for any collection of non-negative integers $\bk=(k_b)_b$, the space of sections of the line bundle $\eL_\bk=\bigotimes_{b} \eL_b^{\T k_b}$ is identified
with the dual of the integrable highest weight $\widehat{\msl}_n$ module $L(\Lambda_\bk)$, where 
the weight $\Lambda_\bk$ is equal to $\sum k_b\Lambda_b$, $\Lambda_b$ are the affine fundamental weights.  
The level of the weight $\Lambda_\bk$ is equal to $|\bk|=\sum k_b$. 
The Demazure modules $D_w(\Lambda_\bk)\subset L(\Lambda_\bk)$ are representations of the Iwahori Lie
algebra $\fI$; they are identified with the dual space of sections $H^0(Y_w,\eL_\bk)^*$.

The coherence conjecture proved by Zhu states that 
\begin{equation}\label{eq:cohconj}
\dim H^0\bigl(\bigcup_{w\in\mathcal{A}_\la} Y_w,\eL_\bk\bigr) = \dim H^0(X_\la,\eO(|\bk|)).
\end{equation}  
The equality of dimensions was recently upgraded in \cite{HY24} to the equality (up to a certain
shift) of characters
with respect to the Cartan subalgebra (in fact, as shown in \cite{HY24}, a larger algebra may 
show up). Let us emphasize  that the left-hand side of \eqref{eq:cohconj} is a representation of the 
Iwahori algebra     
and the right-hand side is a representation of the current algebra. 
Both the Zhu proof \cite{Zhu14} and the Hong--Yu construction use line bundles on the global
affine Grassmannian: the right-hand side of \eqref{eq:cohconj} is the space of sections
on the general fiber, and the left-hand side is the space of sections on the special fiber.
Hence, \eqref{eq:cohconj} can be seen as a degeneration of the space in the right-hand side to 
the space in the left-hand side. 
Our main goal is to give an algebraic description of this degeneration procedure. To this end, we 
introduce a more general construction which leads to the desired description in certain special
cases.

Our construction starts with a collection of $n$ cyclic representations $D_b$,
$b=0,\dots,n-1$ of $\msl_n[z]$ with cyclic vectors $d_b$.
The output is a (no longer cyclic) representation  $\bfD(0)$ of the Iwahori algebra.
The character of $\bfD(0)$ is equal to the character of the Cartan component $D$  
inside the tensor product of the modules $D_b$: $D = \U(\msl_n[z]).\bigotimes d_b$.
The module $\bfD(0)$ can be seen as a degeneration of $D$. 
Let us describe our construction in detail for $D_b$ being irreducible $\msl_n$
modules  with the trivial action of $z\msl_n[z]$ (for the general case see section \ref{sec:Dem}). 

We start with an integral dominant $\msl_n$ weight $\la$ and a decomposition $\la=\sum_b \la^{(b)}$
into a sum of $n$ weights $\la^{(b)}$. 
Then the irreducible $\msl_n$ module $V_\la$ sits inside the tensor product of $V_{\la^{(b)}}$ as a Cartan component.
Let $v_b\in V_{\la^{(b)}}$ be the lowest weight cyclic vectors, so each module is generated from $v_b$ 
by the action of the Chevalley generators
$e_i\in\msl_n$, $i=1,\dots,n-1$ corresponding to the simple roots. 
We deform $V_\la$ inside the tensor product 
$\bigotimes_b V_{\la^{(b)}}$.
The deformation is defined via the operators  $e_i(\veps)= \veps e_i^{(i)}+\sum_{b\ne i} e_i^{(b)}$,
where  $e_i^{(b)}$ is the operator acting as $e_i$ at the $b$-th factor of the tensor product 
and as identity on all other factors. Then we define the space $V_\barla(\veps)$ inside 
$\bigotimes_b V_{\la^{(b)}}$ as the subspace generated from the tensor product of lowest weight vectors
$v_b$   by the action of operators $e_i(\veps)$. Let $V_\barla=V_\barla(0)$ be the $\veps\to 0$ limit of the spaces
$V_\barla(\veps)$. We show that $V_\barla$ carries a natural action of (certain quotient of) the 
Iwahori algebra $\fI$.

\begin{theoremA}
If all $\la^{(b)}$ are multiples of a single fundamental weight $\omega_j$, i.e. $\la^{(b)}=k_b\om_j$, 
then $V_\barla$ is isomorphic to a sum of affine Demazure modules inside $L(\sum_{b} k_b \Lambda_b)$. 
\end{theoremA}
 
\begin{conjA}\label{conj:intro}
For any collection of weights $\barla$ the Iwahori algebra module $V_\barla$  is generated from the tensor products of extremal
weight vectors.
\end{conjA}

We are not able to prove the conjecture, but we have performed numerous checks using the program 
found in Appendix \ref{sec:app}. 
We note that the Iwahori algebra acts on $V_\barla$ via a certain quotient, 
called Inonu--Wigner contraction (see \cite{PY12,F-M25}),
which is a close cousin of the
Drinfeld double of the Borel subalgebra \cite{BNvdV20,BR24,KZ07}. 
This quotient can also be realized in two other ways:
as a degeneration of the Lie algebra $\msl_n$ \cite{F11,F12,FFL11}, and as an endomorphism algebra of certain
representation of the cyclic equioriented quiver \cite{FLP23,FLP24,Kn08}. 

In order to describe the $\fI$ modules $V_\barla$ in general, we consider the 
Kostant--Kumar modules (see \cite{Kum88,Kum89,KRV24}). Let us fix a collection $\bar\Lambda$ of affine integrable weights $(\Lambda_b)_b$ and a collection $\bar w$ of affine Weyl group elements $(w_b)_b$; let 
$d_b\in  L(\Lambda_b)$ be the extremal vectors of weight $w_b(\Lambda_b)$. 
We define $K({\bar w},\bar\Lambda)\subset \bigotimes_b L(\Lambda_b)$ as $\U(\fI).\bigotimes_b d_b$. 

Let $V_\barla'\subset V_\barla$ be the Iwahori algebra subrepresentation inside the tensor product 
of $V_{\la^{(b)}}$ generated from the tensor product of extremal vectors (if the conjecture above is
true, then $V_\barla'=V_\barla$). We show that the $\fI$-module $V_\barla'$ is isomorphic to 
a direct sum of $|W|$ many affine Kostant--Kumar modules \cite{Kum88,KRV21}.

Let us close with two remarks. First, the algebraic degeneration of $V_{\sum \la^{(b)}}$ to 
$V_\barla$  has a geometric counterpart: the flag variety degenerates to a union of affine 
Kostant--Kumar Schubert varieties. We expect that Conjecture \ref{conj:intro} is related to 
the flatness of this degeneration.
Second, all the constructions above work well in the generality
of arbitrary spherical Demazure modules $D_{\la}$ instead of the irreducible $\msl_n$ modules
$V_\la$; the 
corresponding formalism can be found in section \ref{sec:Dem}. 
However, the corresponding Iwahori algebra module is, in general,not generated from the extremal 
weight vectors (see examples and non-examples in section \ref{sec:examples}).
 
The paper is organized as follows. In Section \ref{sec:prelim}, we collect main 
finite-dimensional and infinite-dimensional objects used in the paper. In Section \ref{sec:cohconj},
we formulate the Zhu theorem proving the Pappas--Rapoport coherence conjecture. 
In Section \ref{sec:globalag}, the lattice formalism for the type A affine Grassmannians and 
flag varieties is recalled.	 Section \ref{sec:construction} contains the main construction of
our paper, which produces an Iwahori algebra module starting from a collection of spherical
Demazure modules. In Section \ref{sec:fin-dim}, we consider the special case where the Demazure modules
are irreducible finite-dimensional representations of $\msl_n$ and in Section \ref{sec:Dem} 
we treat a more general case. Section \ref{sec:examples} contains several explicit examples
of our general construction. Finally, in Appendix \ref{sec:app}, we present the computer program 
used to verify special cases of Conjecture \ref{conj:intro}.

\subsection*{Acknowledgments}
EF and AK were partially supported by ISF grant 493/24.

\section{Preliminaries}\label{sec:prelim}
\subsection{Finite-dimensional algebras}
For $\fg=\msl_n$ let us fix the Cartan decomposition $\fg=\fb\oplus\fn_-$, $\fb=\fh\oplus\fn$,
where $\fh$ is the diagonal Cartan subalgebra and $\fn$ (resp., $\fn_-$) are the upper- (resp., lower-) triangular
subalgebras. Let $\Phi=\Phi_+\sqcup\Phi_-\subset \fh^*$ be the set of roots decomposed into the disjoint 
union of 
positive and negative roots. We denote by $\al_i\in\Phi_+$, $i=1,\dots,n-1$ the simple roots for $\fg$ and 
by $f_i=f_{\al_i}\in\fn_-$, $e_i=e_{\al_i}\in\fn$ the corresponding Chevalley generators. Explicitly, $f_i=E_{i+1,i}$, $e_i=E_{i,i+1}$ where 
$E_{\bullet,\bullet}$ are matrix units. For a positive root $\al$ we denote by $f_\al\in\fn_-$ the Cartan
generator of weight $-\al$ and by $e_\al\in\fn$ the Chevalley generator of weight $\al$. The highest 
root $\theta$ is equal to the sum of all simple roots $\al_i$. 
The Weyl group $W\simeq S_n$ is generated by simple reflections $s_i$, $i=1,\dots,n-1$ corresponding to
simple roots $\al_i$; we denote by $w_0$ the longest element in $W$.

Let $\om_i$, $i=1,\dots,n-1$ be the fundamental weights and let $P$ be the weight lattice containing 
the cone of integral fundamental weights $P_+=\bigoplus_i \bZ_{\ge 0} \om_i$.  
The lattice $P$ contains the root lattice $Q$ generated by the simple roots $\al_i$.
Let $(\cdot,\cdot)$ be the standard invariant Killing form on $\fh^*$. 
The simple coroots $\al_i^\vee\in\fh$ are defined by $(\al_i,\beta)=\beta(\alpha_i^\vee)$.
The simple coroots $\al_i^\vee$ generate the coroot lattice $Q^\vee$ and the fundamental
coweights $\om_i^\vee$ generate the coweight lattice $P^\vee$. 

The irreducible highest weight modules of $\msl_n$ are denoted by $V_\mu$ with $\mu\in P_+$. In each $V_\mu$
we fix a highest weight vector and denote it by $v_\mu$. The extremal weights of $V_\mu$ are the $W$-shifts
$\sigma\mu$ of $\mu$; for each $\sigma\in W$ let $v_{\sigma\mu}\in V_\mu$ be a fixed extremal weight vector
spanning its weight space. The Demazure modules $D_{\sigma\mu}\subset V_\mu$ are defined as 
$\U(\fb)v_{\sigma\mu}$. In particular, $D_{w_0\mu}=V_\mu$ and $D_\mu$ is a one-dimensional space spanned by the highest weight
vector $v_\mu$.

In what follows we work with both the special linear Lie algebra $\msl_n$ and the general linear Lie algebra $\mgl_n$. 
The Cartan (diagonal) algebra of $\gl_n$ if $n$-dimensional; the coweights (in the 
standard basis) are denoted by $\la=(\la_1,\dots,\la_n)$, in particular, the 
compositions $\la$ form the lattice of integral coweights. 
Let $\bX\simeq \bZ^n$ be coweight lattice; we denote by $\om_i^\vee\in\bX$ the elements given by
$(1,\dots,1,0,\dots,0)$ (with $i$ units).
Hence every fundamental coweight of $\msl_n$ can be seen inside $\bX$ and we 
get an embedding of the $\msl_n$ coweight lattice $P^\vee$  into $\bX$. In particular,
the image of the dominant cone $P^\vee_+$ is identified with the set of integral dominant coweights $\la$ ($\la_i\ge \la_{i-1}$ for all $i$) such that $\la_n=0$.

Finally, let us note that the irreducible highest weight $\mgl_n$ modules $V_\mu$ are labeled by dominant compositions $\mu$; we note that the overall shift of all components of $\mu$ does not change $V_\mu$ as a module over the embedded $\msl_n\subset\mgl_n$.

\begin{rem}
Since we are working only in type $A$, weights and coweights (roots and coroots) 
are naturally identified; in what follows we sometimes use the same symbols to denote
the elements of $\fh$ and the elements of $\fh^*$.	
\end{rem}   

\subsection{Infinite-dimensional algebras}
For a Lie algebra $\fg$ the corresponding current algebra $\fg[z]$ is defined as $\fg[z]=\fg\T\bC[z]$. It contains a finite-dimensional subalgebra 
$\fg\simeq \fg\T z^0$ and an infinite-dimensional Iwahori subalgebra $\fI=\fb[z]\oplus z\fn_-[z]$. 
The Iwahori algebra is generated by the elements $f_i=f_i\T 1$, $i=1,\dots,n-1$, by the element $f_0=e_\theta\T z$ and by the Cartan subalgebra.

\begin{rem}
Let $\bK=\bC((z))$ and $\bO=\bC[[z]]$ be the field of formal Laurent series and the 
ring of formal power series. The current algebra $\fg[z]$ admits a completion 
$\fg[[z]]=\fg(\bO)$. The representations of the current algebra we consider in this paper
are graded and finite-dimensional, i.e. there exists a number $N$ such that 
$\fg\T z^N\bC[z]$ acts trivially. Therefore, all the $\fg[z]$ modules we consider 
are also the $\fg(\bO)$ modules and we freely switch between the two.  
\end{rem}

The Iwahori subalgebra plays a role of a Borel subalgebra inside the affine Kac--Moody Lie algebra 
$\gh=\fg\T \bC[z,z^{-1}]\oplus\bC K\oplus\bC d$, where $K$ is central and $d$ is the degree operator.
The affine Cartan subalgebra $\fh^a\subset \gh$ is spanned by $\fh\T 1$ and the elements $K$, $d$.  
The dual Cartan subalgebra $(\fh^a)^*$ is spanned by the roots $\al_i$, $i=1,\dots,n-1$ 
(which vanish at $K$ and $d$) and the elements $\Lambda_0$, $\delta$ such that 
$\Lambda_0(\al_i)=\delta(\al_i)=0$ for $1\le i<n$, $\Lambda_0(K)=1=\delta(d)$, $\Lambda_0(d)=\delta(K)=0$.
The affine simple root $\al_0$ is equal to $\delta-\theta$.

An affine dominant integral weight $\Lambda\in(\fh^a)^*$ is defined by its restriction to the 
finite part
$\Lambda|_\fh$ and by the level $\Lambda(K)$ (we will always assume that $\Lambda(d)=0$). The level
is a non-negative integer and one has the restriction $(\Lambda|_\fh,\theta^\vee)\le \Lambda(K)$.
For $\fg=\msl_n$ there are $n$ level one integral dominant weights: the possible finite parts
are fundamental weights or zero. We denote the lattice of affine dominant integral weights of $\widehat{\msl}_n$ by $P^a_+$
and the representation corresponding to $\Lambda\in P^a_+$ by $L(\Lambda)$. The space $L(\Lambda)$ 
contains a highest weight vector $v_\Lambda$ such that $L(\Lambda)$ is generated from $v_\Lambda$
by the action of $\fb_-[z^{-1}]\oplus z^{-1}\fn[z^{-1}]$. 

The affine type $A^{(1)}_{n-1}$ Dynkin diagram (corresponding to the Lie algebra $\widehat\msl_{n}$)  is the cyclic
graph $\Delta$ with $n$ vertices. We denote the set of vertices by $\Delta_0$ and identify 
$\Delta_0$ with the set of 
numbers $b\in\bZ/n\bZ$; we also assume that the graph $\Delta$ is equi-oriented. 
Let $W^a$ be the affine Weyl group. The group $W^a$ is generated by the simple reflections $s_b$,
$b\in\Delta_0$. It can be explicitly written as $W^a=W\ltimes Q^\vee$, where $Q^\vee$ is the coroot
lattice. 
For $\la\in Q^\vee$ we denote the corresponding
element of $W^a$ by $z^\la$. 

For $b\in \Delta_0$ let $\Lambda_b\in(\fh^a)^*$  be the affine level one fundamental
weights: the finite part of $\Lambda_0$ is zero, $\Lambda_b|_\fh=\omega_b$ for $b\ne 0$.
The weight $\Lambda_0$ is called basic (or vacuum) level one weight.
Any integral dominant weight $\Lambda$ can be written as 
$\sum_{b\in\Delta_0} m_b\Lambda_b$ with some non-negative integer coefficients $m_b$.
        
For an element $\sigma\in W^a$ let $v_{\sigma\Lambda}\in L(\Lambda)$ be a fixed extremal 
vector of weight  $\sigma\Lambda$. The affine Demazure module 
$D_\sigma(\Lambda)\subset L(\Lambda)$ is equal to 
$\U(\fI)v_{\sigma\Lambda}$, i.e. is generated from the extremal weight vector by the 
action of the Iwahori algebra. 

We also need the $\msl_n$ and $\mgl_n$ extended affine Weyl groups. 
Recall that the type $A_{n-1}^{(1)}$
extended affine Weyl group is the semi-direct product of the finite Weyl group $W$ and the coweight lattice $P^\vee$. The $\mgl_n$ version is the semi-direct product $W_{\rm ext}=S_n\ltimes \bX$.

\section{The coherence conjecture}\label{sec:cohconj}
Recall the notation $\bK=\bC((z))$ and $\bO=\bC[[z]]$ for the field of formal  Laurent 
series and its subring 
of formal power series.  For a simple Lie algebra $\fg$ let  
$G$ be the corresponding simply-connected Lie group; in this paper we only consider the case $\fg=\msl_n$, $G=SL_n$. 
For $b\in\Delta_0=\bZ/n\bZ$ let $\aGr_b$ be the corresponding affine Grassmannian, which is the quotient 
of $G(\bK)$ by the $b$-th maximal parabolic subgroup. In particular, 
$\aGr_0$ is the quotient of $G(\bK)$ by the current group $G(\bO)=G[[z]]$. 
All the affine Grassmannians are ind-varieties, i.e. inductive limits of finite-dimensional 
subvarieties. 
One has a $\widehat G$-equivariant embedding 
$\aGr_b\subset\bP(\Lambda_b)$; we denote by $\eO$ the pull-back from $\bP(\Lambda_b)$ of 
the line bundle $\eO(1)$. The line bundle $\eO$ is known to be a generator of the Picard group of the 
affine Grassmannian $\aGr_b$.

The Schubert varieties inside affine Grassmannians $\aGr_b$ are defined as the closures of the
Iwahori group orbits through the torus fixed points. We will only be interested in the spherical 
Schubert varieties in the affine Grassmannians, i.e. whose Schubert varieties which are invariant 
with respect to the action of the current group $SL_n(\bO)$.
The spherical Schubert varieties are labeled by dominant coweights $\la$;
the corresponding Schubert variety is denoted by $X_\la$
(since we are working in type $A$, we  sometimes identify coweight lattice with the weight lattice).
For a fixed $\la$ the variety $X_\la$ sits inside $\aGr_b$ such that $\la-\om_b^\vee$ is in the coroot lattice (if $\la$ itself belongs to the coroot lattice, then $b=0$).

\begin{rem}
Sometimes it is more convenient to work with all the affine Grassmannians together. To this end, one starts with an adjoint group $G_{\rm ad}$ 
and consider the quotient $G_{\rm ad}(\bK)/G_{\rm ad}(\bO)$. This quotient
is identified with the disjoint union over all $b$ of the affine Grassmannians
$\aGr_b$. Then the line bundles $\eO$ can be glued into a single line bundle on   $G_{\rm ad}(\bK)/G_{\rm ad}(\bO)$.
\end{rem} 

Let $\aFl=G(\bK)/\bI$ be the affine flag variety, where $\bI$ is the Iwahori subgroup --
the preimage of the Borel subgroup of $G$ with respect to the $z=0$ evaluation map. 
As in the case of affine Grassmannians, 
$\aFl$ is an ind-variety as an inductive limit of finite-dimensional Schubert subvarieties (for more details see below). 
Using the projections  $\aFl\to \aGr_b$, one
constructs line bundles $\eL(\Lambda_b)$ on $\aFl$ as pull backs of $\eO$. The line bundles 
$\eL(\Lambda_b)$, $b\in\Delta_0$ generate the Picard group of the affine flag variety. For
$\Lambda=\sum_b m_b\Lambda_b$ we denote by $\eL(\Lambda)$ the line bundle 
$\bigotimes_{b\in\Delta_0} \eL(\Lambda_b)^{\T m_b}$. One can also define partial affine
flag varieties, interpolating between $\aFl$ and $\aGr_b$, but in this paper we need only the complete 
affine flags.

The affine flag variety contains the set of torus fixed points $p_\sigma$ for $\sigma\in W^a$.
The Iwahori group orbits $\bI.p_\sigma$ are called Schubert cells and the closures are called Schubert
varieties. We denote $\overline{\bI.p_\sigma}$ by $Y_\sigma\subset \aFl$. 
Each $Y_\sigma$ is a finite-dimensional  
(in general, singular) projective algebraic variety; the union of all Schubert varieties coincides with $\aFl$. 
One has the following important property of the Schubert varieties: the space of sections $H^0(Y_\sigma,\eL(\Lambda))$ is isomorphic (as the Iwahori algebra module) to the dual of an affine Demazure module.

\begin{rem}\label{rem:nonconfl}
Similar to the situation with affine Grassmannian, one can consider a non-connected version of the affine flags (with the connected components labeled by the quotient of the coweight lattice by the coroot lattice). 
The Iwahori algebra orbits are then parametrized by the elements of the extended affine Weyl group (as opposed to the smaller affine Weyl group as above).   
\end{rem}

In order to formulate the Zhu theorem (Pappas--Rapoport conjecture) we need one more piece of notation.
For a coweight $\la$ let $z^\la$ be the corresponding extended affine Weyl group element. Let  
\begin{equation}\label{eq:Ala}
\mathcal A(\la) = \bigcup_{w\in W} Y_{z^{w\la}}\subset \aFl,
\end{equation}
i.e. $\mathcal A(\la)$ is a union of (closed) Schubert varieties inside the affine flag variety.

\begin{rem}
In \eqref{eq:Ala} we consider arbitrary coweights $\la$, hence $z^\la$ does not necessarily belongs to the affine Weyl group, but to the extended affine Weyl group. Therefore the use of the non-connected affine flags 
is necessary in this formulation (see Remark \ref{rem:nonconfl}).
However, one can stay with the standard connected affine flag variety 
by replacing the right hand side of \eqref{eq:Ala} by the union of Schubert
varieties corresponding to the elements $\sigma\in W^a$ such that $\sigma$
is smaller than some element of the form $z^{w\la}$, $w\in W$. 
We do not go into details here, since in type $A$ the whole picture can (and will) be made very explicit using the $GL_n$ lattice formalism.  
\end{rem}

Let $\Lambda=\sum_{b\in\Delta_0} m_b \Lambda_b$ be an affine dominant integral weight.
In particular, the level $\Lambda(K)$ of $\Lambda$ is equal to the sum of all $m_b$. Recall the
line bundles $\eL(\Lambda)$ on $\aFl$ and the line bundle $\eO$ on the affine Grassmannians. 
The coherence conjecture states that 
\begin{equation}
\dim H^0(X_\la,\eO^{\T \Lambda(K)}) = \dim H^0(\mathcal A(\la),\eL(\Lambda)).
\end{equation}       
The equality of dimensions was upgraded to the isomorphism of modules over the Cartan subalgebra by Hong
and Yu (see \cite{HY24}).

\section{Global affine Grassmannians in type A}\label{sec:globalag}
\subsection{Lattices}
The major role in the proof of the coherence conjecture is played by the global affine Grassmannians,
which is a family connecting affine Grassmannians and affine flag varieties 
(see \cite{Zhu14,HY24,HaN02}).
In type $A$ one can make the construction explicit using the lattice formalism (see e.g. \cite{AB24,AR,Zhou19}).
Let us recall the setup.

Let $w_1,\dots,w_n$ be a standard basis of the $n$-dimensional vector space. 
A lattice $L$ is a subspace in $\bK^n=\mathrm{span}\{w_i\}_{i=1}^n\T \bK$ which is a free $\bO$ module 
of rank $n$. For example, for a coweight
$\la=(\la_1,\dots,\la_n)\in\bZ^n$ we denote by $\bfL^\la$ the lattice generated by $z^{-\la_i}w_i$, 
$1\le i\le n$. The charge $\nu(\bfL^\la)$ is equal to the sum of all $\la_i$: 
$\nu(\bfL^\la)=|\la|$. 
For example, for $\om_b=(1,\dots,1,0,\dots,0)$ (with $b$ units) 
the lattice $\bfL^{\om_b}$ is equal to 
$\bO^n\oplus z^{-1}\mathrm{span}\{w_1,\dots,w_b\}$ and $\nu(\bfL^{\om_b})=b$.
For a general lattice
$L$ with an $\bO$ basis $u_1,\dots,u_n$ its charge $\nu(L)$ is defined as the negated smallest
$z$-degree showing up in  the determinant of the matrix whose columns are the $n$-vectors $u_i$ (with coefficients in $\bK$). The group $SL_n(\bO)$ acts on the space of lattices and preserves the charge.

The affine Grassmannian $\aGr_b$ is realized as the space of lattices $L$ such that $\nu(L)=b$;
in particular, $\aGr_b\ni \bfL^{\om_b}$ (including the case $\om_0=0$).
The affine Grassmannian $\aGr_b$ is isomorphic
to the quotient of $SL_n(\bK)$ by the $b$-th maximal parabolic subgroup. In general, 
all lattices realize the $GL_n$ affine Grassmannian $GL_n(\bK)/GL_n(\bO)$. We note that the multiplication
by $z^{-1}$ adds $n$ to the charge of a lattice, so in principal one can consider affine Grassmannians $\aGr_b$
for any integer $b$, consisting of lattices of charge $b$, but $\aGr_b$ is naturally identified with
$\aGr_{b+n}$. 

The affine flag variety $\aFl$ sits inside the product of affine Grassmannians: 
$\aFl\subset \prod_{b\in\Delta_0} \aGr_b$.
Explicitly, $\aFl$ consists of collections $(L_b)_{b}$ such that 
\[
L_0\subset L_1\subset\dots\subset L_{n-1}\subset z^{-1}L_0,\ \dim L_{i+1}/L_i=1.
\]
In particular,  the base point of $\aFl$ is represented by the chain $(\bfL^{\om_b})_{b=0}^{n-1}$. 
For $w=(\sigma,\la)\in W_{\rm ext}$, $\sigma\in S_n$, $\la\in\bX$ we define the point $\bfL^w\in\aFl$ by 
\begin{equation}\label{eq:Lw}
\bfL^w=\left(\bfL^{\sigma(\la)}\subset \bfL^{\sigma(\la+\om_1)}\subset \dots\subset  
\bfL^{\sigma(\la+\om_{n-1})}
\subset z^{-1}\bfL^{\sigma(\la)}\right).
\end{equation}
As mentioned above, the spherical Schubert varieties in the affine Grassmannian are labeled by 
dominant coweights $\la$; the variety $X_\la$  is the closure of the orbit $SL_n(\bO).\bfL^\la$.
The spherical Schubert varieties sitting inside $\aGr_b$ correspond to $\la$ with $|\la|=b$. 
The Schubert subvarieties inside the affine flag variety $\aFl$ are the closures of
the Iwahori group orbits of the points $\bfL^w$, $w\in W_{\rm ext}$.

\subsection{Global affine Grassmannian}\label{subsec:globalaffine}
Let us describe a family over $\bA^1$ whose general fiber is an affine Grassmannian
$\aGr_b$ 
and the special fiber is isomorphic to the affine flag variety. For simplicity,
the family below corresponds to the case $b=0$, the general case does not differ much.

We introduce the $\bK$-linear maps $\psi_b: \bK^n\to \bK^n$, $b\in\bZ/n\bZ$ by 
\begin{equation}\label{eq:psib}
\psi_b(w_i)=\begin{cases} w_i, & i\ne b+1,\\ (z+\veps)w_i, & i=b+1\end{cases}, 
\end{equation}
where $\veps$ is a complex number.
The global Grassmannian $\bf{Gr}$ is formed by collections $(\veps,L_0,\dots,L_{n-1})$ such that
$\veps\in\bC$, $L_i\in\aGr_0$ and $\psi_{b} L_i\subset L_{i+1}$ for all $i=0,\dots,n-1$
(the condition for $i=n-1$ reads as $\psi_{n-1} L_{n-1}\subset L_0$).
We note that the composition of all the maps $\psi_{b}$ is equal to $(z+\veps)\mathrm{Id}$, 
the condition 
$(z+\veps)L_0\subset L_0$ does hold, since $L_0$ is a lattice. 
The general fiber of the natural map $\pi:\bf{Gr}\to\bC$ (projection to the first coordinate $\veps$)
is isomorphic to the affine Grassmannian $\aGr_0$, since $z+\veps$ is invertible
in $\bO$ for $\veps\ne 0$. and the special fiber (over $\veps=0$)
is isomorphic to the affine flag variety ${\aFl}$; the isomorphism is given by 
$(L_i)_i\mapsto (A^iL_i)$, where $A$ is the natural shift operator which 
identifies $\aGr_\bullet$ with $\aGr_{\bullet+1}$.

The spherical Schubert varieties $X_\la$ admit a global version $\bf{X}_\la\subset \bf{Gr}$ \cite{AB24}. The general
fiber of the restriction of the projection map $\pi$ to $\bf{X}_\la$ is isomorphic to the Schubert variety $X_\la$ and the special fiber 
$\pi^{-1}(0)\cap \bf{X}_\la\subset {\mathcal Fl}$ is equal to a union of Schubert varieties 
$\mathcal A_\la$ \eqref{eq:Ala}.
One can describe the components explicitly using the Kottwitz-Rapoport alcoves \cite{KR00}. 

\section{The construction}\label{sec:construction}
In this section we formulate the general form of the algebraic version of the
Pappas--Rapoport construction. In the following sections
we discuss the details and describe certain special cases (in particular, related to the coherence conjecture).

\begin{lem}\label{lem:globK^n}
Let us consider the representation $M$ of the cyclic equioriented quiver $\Delta$ such 
that all the spaces $M_b$ are isomorphic to $\bK^n$ and the
map from the vertex $b$ to $b+1$ is $\psi_b$ \eqref{eq:psib}. Then the endomorphism algebra $\mathrm{End}_\Delta(M)$ consists of collections of
$\bK$-linear maps $(A_b)_{b\in \Delta_0}:\bK^n\to\bK^n$ such that 
\[
(A_{b+1})_{i,j} = \begin{cases} 
	(z+\veps)(A_b)_{i,j}, & i=b+1, j\ne b+1\\ (z+\veps)^{-1}(A_b)_{i,j}, & j=b+1, i\ne b+1\\
	(A_b)_{i,j}, & \text{ otherwise} 
	\end{cases}
\]     	
\end{lem}
\begin{proof}
Direct computation.
\end{proof}

Here is an example for $n=4$:
\[
A_0= \left(\begin{smallmatrix}
a_{11} & a_{12} & a_{13} & a_{14}\\
a_{21} & a_{22} & a_{23} & a_{24}\\
a_{31} & a_{32} & a_{33} & a_{34}\\
a_{41} & a_{42} & a_{43} & a_{44}
\end{smallmatrix}\right),\
A_1= \left(\begin{smallmatrix}
a_{11} & a_{12}(z+\veps) & a_{13}(z+\veps) & a_{14}(z+\veps)\\
\frac{a_{21}}{z+\veps} & a_{22} & a_{23} & a_{24}\\
\frac{a_{31}}{z+\veps} & a_{32} & a_{33} & a_{34}\\
\frac{a_{41}}{z+\veps} & a_{42} & a_{43} & a_{44}
\end{smallmatrix}\right),
\]
\[
A_2= \left(\begin{smallmatrix}
	a_{11} & a_{12} & a_{13}(z+\veps) & a_{14}(z+\veps)\\
	a_{21} & a_{22} & a_{23}(z+\veps) & a_{24}(z+\veps)\\
	\frac{a_{31}}{z+\veps} & \frac{a_{32}}{z+\veps} & a_{33} & a_{34}\\
	\frac{a_{41}}{z+\veps} & \frac{a_{42}}{z+\veps} & a_{43} & a_{44}
\end{smallmatrix}\right),\
A_3= \left(\begin{smallmatrix}
	a_{11} & a_{12} & a_{13} & a_{14}(z+\veps)\\
	a_{21} & a_{22} & a_{23} & a_{24}(z+\veps)\\
	a_{31} & a_{32} & a_{33} & a_{34}(z+\veps)\\
	\frac{a_{41}}{z+\veps} & \frac{a_{42}}{z+\veps} & \frac{a_{43}}{z+\veps} & a_{44}
\end{smallmatrix}\right).
\]

By definition, the endomorphism algebra is a subalgebra of $\bigoplus_{b\in\Delta_0}\mgl_n(\bK)$.
Its $z$-non-negative part -- the intersection  with the $\bigoplus_{b\in\Delta_0}\mgl_n(\bO)$ -- 
is a free $\bO$ module 
generated by the diagonal part
\begin{equation}
\label{eq:diag}	
(E_{i,i},\dots,E_{i,i}),\ 1\le i\le n,
\end{equation}
the upper-triangular part
\begin{equation}
\label{eq:uppertr}
(\underbrace{E_{i,j},\dots,E_{i,j}}_i, 
\underbrace{(z+\veps)E_{i,j},\dots,(z+\veps)E_{i,j}}_{j-i},
\underbrace{E_{i,j},\dots,E_{i,j}}_{n-j}) 
\end{equation}
for all $1\le i <j \le n$, and the lower-triangular part
\begin{equation}
\label{eq:lowertr}
(\underbrace{(z+\veps)E_{j,i},\dots,(z+\veps)E_{j,i}}_i, 
\underbrace{E_{j,i},\dots,E_{j,i}}_{j-i},
\underbrace{(z+\veps)E_{j,i},\dots,(z+\veps)E_{j,i}}_{n-j}) 
\end{equation}
for all $1\le i <j \le n$.
We denote the Lie algebra generated by the elements above by $\fa(\veps)$.

\begin{lem}\label{lem:aIw}
The Lie algebra $\fa(0)$ is isomorphic to the Iwahori algebra $\fI$. For $\veps\ne 0$ the Lie algebra $\fa(\veps)$ is isomorphic to $\mgl_n(\bO)$.
\end{lem}
\begin{proof}
We first note that for any $\veps$ the $\bO$-span of the elements \eqref{eq:uppertr} and \eqref{eq:diag} is the Lie algebra $\fb(\bO)$ of upper-triangular matrices in $\mgl_n(\bO)$.
Now for $\veps=0$ the elements \eqref{eq:lowertr} add $z\fn_-(\bO)$ .
For $\veps\ne 0$, since $z+\veps$ is invertible in $\bO$,  \eqref{eq:lowertr} divided by 
$z+\veps$ together with \eqref{eq:diag} and \eqref{eq:uppertr} form a $\bO$ basis for 
$\mgl_n(\bO)$.
\end{proof}

Let $D$ be a cyclic graded finite-dimensional $\msl_n(\bO)$-module with cyclic vector $v$ 
(one can replace $\msl_n(\bO)$ with $\msl_n[z]$). We assume that 
\begin{itemize}
	\item $v$ is a weight vectors, $hv=\la(h)v$, $h\in\fh$ for some weight $\la\in \fh^*$,
	\item $\fn_-[z] v=0$ and $z\fh[z]v=0$,
	\item $\U(\fn[z])v=D$.  
\end{itemize}
In particular, $D$ is cyclic as the Iwahori algebra module.
Recall the Chevalley generators $e_b\in\fn$, $e_b=E_{b,b+1}$.
For a complex number $\veps\ne 0$ and $b=1,\dots,n-1$ let $\mathrm{sh}_b:\fn(\bO)\to \fn(\bO)$ 
be a Lie algebra endomorphism defined by 
\begin{equation}\label{eq:sheb}
e_bz^k \mapsto e_bz^k(z+\veps),\qquad e_{a}z^k \mapsto e_az^k, a\ne b.
\end{equation}
Explicitly, for $1\le i<j\le n$ and $k\ge 0$ one has:
\begin{equation}\label{eq:sh}
\mathrm{sh}_b (E_{i,j}z^k)=\begin{cases}
\veps E_{i,j}z^k+ E_{i,j}z^{k+1}, & i\le b<j,\\
E_{i,j}z^k, & \text{ otherwise}.
\end{cases}
\end{equation}
One easily checks that $\mathrm{sh}_b$ is indeed compatible with the Lie bracket.

\begin{dfn}\label{def:badm}
For $b=1,\dots,n-1$ we say that $D$ is $b$-admissible if for any $\veps\ne 0$ 
there exists a linear isomorphism $f:D\to D$, $f(v)=v$ such that 
for any $x\in\fn(\bO)$ one has $f\circ x = \mathrm{sh}_b(x)\circ f$.  
\end{dfn}

In other words,  $D$ is isomorphic to the $\fn(\bO)$
module obtained from $D$ by shifting the action using the endomorphism $\mathrm{sh}_b$. 

\begin{rem}
We show below that all affine Demazure modules are $b$-admis\-sible for all $b$.	
\end{rem}

\begin{rem}
The image of the embedding $\fa(\veps)\subset\bigoplus_{b\in\Delta_0} \mgl_n(\bO)$ 
contains the upper-triangular subalgebra consisting of elements of the form 
\[
(x,\mathrm{sh}_1(x),\dots, \mathrm{sh}_{n-1}(x)), \ x\in\fn(\bO).
\]
\end{rem}

Now let $D_b$, $b\in\Delta_0=\bZ/n\bZ$ be a collection of $b$-admissible $\msl_n(\bO)$ modules with cyclic
vectors $v_b$ (for $b=0$ the $b$-admissibility condition is empty).
We consider the Cartan component in the tensor product of all modules $D_b$, namely
\[
\odot_{b\in\Delta_0} D_b = \U(\msl_n(\bO)).\otimes_b v_b \subset \bigotimes_{b\in\Delta_0} D_b.
\]
The tensor product $\bigotimes_{b} D_b$ is acted upon by the algebra $\fa(\veps)$ via the embedding
$\fa(\veps)\to \bigoplus_b \mgl_n(\bO)$. For $\veps\ne 0$ we define
\begin{equation}\label{eq:Dveps}
\bfD(\veps) = \U(\fa(\veps)).\otimes_{b\in\Delta_0} v_b\subset \bigotimes_{b\in\Delta_0} D_b.
\end{equation}

\begin{lem}
Assume that $D_b$ is $b$-admissible for all $b\in\Delta_0$. Then  	
for any $\veps\ne 0$ the $\fn(\bO)$ module $\bfD(\veps)$ is isomorphic to $\odot_{b} D_b$. 
\end{lem}
\begin{proof}
Let $f_b:D_b\to D_b$, $f_bv_b=v_b$ be the isomorphism satisfying 
$f_b\circ x = \mathrm{sh}_b(x)\circ f_b$ for any
$x\in\fn(\bO)$.
Then the desired isomorphism is induced by the map $\bigotimes_{b\in\Delta_0} f_b$ with
$f_0=\mathrm{Id}$.
\end{proof}

We define the subspace \[\bfD(0)=\lim_{\veps\to 0} \bfD(\veps)\subset\bigotimes_{b\in\Delta_0} D_b,\]
which is a degeneration of the Cartan component $\odot_{b} D_b$.
Explicitly, the procedure works as follows. Let us consider $\veps$ as an auxiliary variable 
(parameter). Then $\bfD(\veps)$ sits inside $\left(\bigotimes_{b\in\Delta_0} D_b\right)[\veps]$ \eqref{eq:Dveps}. Let us consider the decreasing filtration of the ambient space by the 
subspaces  $\veps^r\left(\bigotimes_{b} D_b\right)[\veps]$, $r\ge 0$.
One gets an induced decreasing filtration on $\bfD(\veps)$ (the intersection of $\bfD(\veps)$
with the filtration in the ambient space)  
\[
\bfD(\veps)=\bfD(\veps)_0 \supset \bfD(\veps)_1\supset \dots. 
\]
 In particular, for each $r\ge 0$ one gets a map
 \begin{equation}\label{eq:vepsto0} 
 \bfD(\veps)_r/\bfD(\veps)_{r+1} \to \bigotimes_{b\in\Delta_0} D_b,\ 
 p\mapsto \frac{p}{\veps^r}|_{\veps=0},
 \end{equation}
$p\in\left(\bigotimes_{b\in\Delta_0} D_b\right)[\veps]$. Then $\bfD(0)$ is spanned by the 
images of all the maps \eqref{eq:vepsto0}.
In particular, one has the following lemma.

\begin{lem}
The dimension of the module $\bfD(0)$ is equal to $\dim \bfD(\veps)$ for 
$\veps\ne 0$.
\end{lem}
\begin{proof}
Let us choose a basis of $\bfD(\veps)$ compatible with the filtration 
$\bfD(\veps)_r$ (as above, $\veps$ is considered as a formal parameter). 
More precisely, we first find a maximal $r$ such that 
$\bfD(\veps)\cap \veps^r\left(\bigotimes_{b} D_b\right)[\veps]$ is non-empty.
We fix a (finite) basis $B_r$ of this intersection and then pass to the intersection $\bfD(\veps)\cap \veps^{r-1}\left(\bigotimes_{b} D_b\right)[\veps]$.
We complete $B_r$ to a basis of this intersection; let us denote the set of added 
elements by $B_{r-1}$ (thus, $B_r\sqcup B_{r-1}$ is a basis of
$\bfD(\veps)\cap \veps^{r-1}\left(\bigotimes_{b} D_b\right)[\veps]$). We then 
pass to $r-2$ and so on until we fix $B_0$. Now let $B=\bigsqcup_{\ell=0}^r B_\ell$. Then, by construction, $B$ has $\dim \bfD(\veps)$ elements and, by definition, the elements 
$\frac{b}{\veps^\ell}|_{\veps=0}$, $b\in B_\ell$, $\ell =0,\dots, r$ form 
a basis of $\bfD(0)$. Hence, $\dim \bfD(0)= \dim \bfD(\veps)$.         
\end{proof}

\begin{lem}\label{lem:D0Iw}
The module $\bfD(0)$ admits a natural action of the Iwahori algebra.
\end{lem}
\begin{proof}
Each space $\bfD(\veps)$ is invariant with respect to the action of the Lie algebra $\fa(\veps)$.
Hence $\bfD(0)$ is $\fa(0)$ invariant. By Lemma \ref{lem:aIw} $\fa(0)\simeq\fI$. 
\end{proof}

The goal of the rest of the paper is to describe the $\fI$ module $\bfD(0)$ in certain special cases.

\section{Finite-dimensional representations}\label{sec:fin-dim}
In this section we consider the simplest class of representations of the current algebras: the 
finite-dimensional representations of $\msl_n$. Even in this very special case the general construction 
produces non-trivial representations of the Iwahori algebra. We start with introducing the one-parameter
family of Lie algebras (a shadow of the general picture) and then describe the corresponding
representation theory.    

\subsection{Lie algebras}
As above, let $\Delta$ be the equioriented quiver with the set of vertices $\Delta_0=\bZ/n\bZ$
and the set of arrows $Q_1=\{b\to b+1, b\in \Delta_0\}$.
Let $M(\veps)=(M_i)_{i\in \Delta_0}$ be a $\Delta$-module of dimension $(n,\dots,n)$ defined as follows.
We identify all spaces $M_i$ with a vector space $\mathrm{span}\{w_j\}_{j=1}^n$. 
Let all the maps $M_{b\to b+1}$ be defined by $w_j\mapsto w_{j}$ for $j\ne b+1$ and 
$w_{b+1}\mapsto \veps w_{b+1}$ (this is the $z=0$ specialization of \eqref{eq:psib}). 
In particular, the rank of the composition of $s$ consecutive maps in $M(0)$ is equal to $n-s$.

\begin{rem}
For $\veps\ne 0$ the $\Delta$ module $M(\veps)$ is isomorphic to the direct sum of $n$ copies 
of the representation of dimension $(1,\dots,1)$ with all maps being identities.
If $\veps=0$, then $M(0)$ is the direct sum of $n$ different nilpotent indecomposable representations
of dimension $(1,\dots,1)$. More precisely, let $U(i;\ell)$, 
$i\in\bZ/n\bZ$, $\ell\ge 0$ be an indecomposable module of total dimension $\ell$ 
supported on vertices $i,i+1,\dots,i+\ell-1$. Then 
$M(0)\simeq\bigoplus_{i\in\bZ/n\bZ} U(i;n)$.
\end{rem}

We denote the endomorphism algebra $\End_\Delta(M(\veps))$ of $M(\veps)$ by $\fI_1(\veps)$.

\begin{lem}\label{lem:juggling}
The Lie algebra $\fI_1(\veps)$ is of dimension $n^2$ for all $\veps$.
For $\veps\ne 0$ it is isomorphic to $\mgl_n$. The Lie algebra $\fI_1(0)$
is isomorphic to the Iwahori algebra quotient $\fI/z\fI$.  
\end{lem}  
\begin{proof}
For $\veps\ne 0$ the operators $M(\veps)_{i\to i+1}$ identify $M_i$ with $M_{i+1}$ and hence the algebra of endomorphisms of the quiver representation $M(\veps)$ is isomorphic to $\mgl_n$.

If $\veps=0$, then one computes the endomorphism algebra explicitly as follows (see \cite{FLP23}). The Lie algebra $\fI_1(0)$ (which we denote in what follows by $\fI_1$)
is embedded into the direct sum over $b\in\Delta_0$ of the Lie algebras $\mgl_n$.
The image is spanned by the following collections $(x_b)_{b\in\Delta_0}$:
\begin{gather*}
i=1,\dots,n:\ x_b=E_{i,i},\\
1\le i<j \le n:\ x_b=0, i\le b<j;\ x_b = E_{i,j} \text{ otherwise},\\
1\le j<i \le n:\ x_b=E_{i,j}, j\le b<i;\ x_b = 0 \text{ otherwise}
\end{gather*}
(see \eqref{eq:diag}, \eqref{eq:uppertr}, \eqref{eq:lowertr}). One easily
checks that the span of these elements is isomorphic to the Iwahori algebra quotient $\fI/z\fI$; in particular, the elements in the second line --
with $i<j$ -- span the Lie algebra isomorphic to the upper-triangular subalgebra of $\mgl_n$ and the elements from the third line pairwise commute.
\end{proof}

\begin{lem}\label{lem:cycsym}
The algebra $\fI_1$ admits an action of the cyclic group $\bZ/n\bZ$, the group of symmetries
of the quiver $\Delta_0$.
\end{lem}
\begin{proof}
The representation $M(\veps)$ admits a natural $\bZ/n\bZ$ symmetry. Namely, let 
$\mathrm{rot}\in\mathrm{End}(\mathrm{span}\{w_j\}_{j=1}^n)$ be the rotation operator 
sending $w_i$ to $w_{i+1}$ (assuming $w_{n+1}=w_1$). Then a generator of $\bZ/n\bZ$
sends $(m_b)_b\in M(\veps)$ to  $(\mathrm{rot}.m_{b-1})_b$. Now since $\fI_1(\veps)$ is the space
of endomorphisms of $M(\veps)$, the cyclic group action on $M(\veps)$ indices the action
on $\fI_1(\veps)$ (for all $\veps$, including $\veps=0$). 
\end{proof}

\begin{example}
Let $n=3$. Then the vertices of the quiver are labeled by the elements of $\bZ/3\bZ=\{0,1,2\}$ and $\fI_1(\veps)$ is spanned by
the elements  
\begin{gather*}
(E_{1,1},E_{1,1},E_{1,1}),\ (E_{2,2},E_{2,2},E_{2,2}),\ (E_{3,3},E_{3,3},E_{3,3}),\\
(E_{1,2},\veps E_{1,2},E_{1,2}),\ (E_{2,3},E_{2,3},\veps E_{2,3}),\ (E_{1,3}, \veps E_{1,3},\veps E_{1,3}),\\
(\veps E_{2,1},E_{2,1},\veps E_{2,1}),\ (\veps E_{3,2},\veps E_{3,2}, E_{3,2}),\ (\veps E_{3,1}, E_{3,1},E_{3,1}).
\end{gather*}	
In particular, the non-diagonal parts of $\fI_1(0)$ are of the form
\begin{gather*}
(E_{1,2},0,E_{1,2}),\ (E_{2,3},E_{2,3},0),\ (E_{1,3}, 0,0),\\
(0,E_{2,1},0),\ (0,0, E_{3,2}),\ (0, E_{3,1},E_{3,1}).
\end{gather*}	
\end{example}

\subsection{Fundamental representations}
From Lemma \ref{lem:juggling} one gets the following explicit description:
$\fI=\fI_1(0)\simeq \fb\oplus(\fn_-)^a$, where $\fb$ is the upper-triangular Borel 
subalgebra in $\mgl_n$ and $(\fn_-)^a$ is the abelian Lie algebra with the underlying vector space 
being the space of strictly lower triangular matrices (see \cite{BR24,FFL11,F12}). The space $(\fn_-)^a$ is an abelian ideal in
$\fI_1$ and the action of $\fb$ on the ideal is induced by the isomorphism 
$(\fn_-)^a\simeq \mgl_n/\fb$.    

\begin{rem}
The Lie algebra $\mgl_n$ contains an identity element, which we sometimes ignore 
and pass to $\msl_n$.
\end{rem}

Recall the embedding $\fI_1\subset \bigoplus_{b} \mgl_n$. The projection to the $b$-th summand
$\fI_1\to\mgl_n$ produces the fundamental
representations $V_{\om_k}^{(b)}$ of $\fI_1$ for all $k=1,\dots,n-1$ and $b=0,\dots,n-1$.
In particular, for all $b$ one has the vector space isomorphism $V_{\om_k}^{(b)}\simeq \Lambda^k(\bC^n)$.

\begin{example}\label{ex:b=0}
For $b=0$ the algebra $\fI_1\simeq \fb\oplus(\fn_-)^a$ acts on $V_{\om_k}^{(0)}$ as follows:
$\fb$ acts in the usual way on $\Lambda^k(\bC^n)$ and $(\fn_-)^a$ acts trivially.
\end{example}
 
Recall the basis $w_1,\dots,w_n$ of the $n$-dimensional ambient vector space $\bC^n$. 
 
\begin{lem}
The $\fI_1$ module $V_{\om_k}^{(b)}$ has a unique (up to a scalar) cyclic vector
$c^{(b)}_k$ and a unique (up to a scalar) cocyclic vector $cc^{(b)}_k$. These vectors are
explicitly given by
\begin{gather*}
c^{(b)}_k = w_b\wedge w_{b-1}\wedge \dots\wedge w_{b-k+1},\\
cc^{(b)}_k =  w_{b+1}\wedge w_{b+2}\wedge \dots\wedge w_{b+k},
\end{gather*}	  
where the indices are taken modulo $n$. The cocyclic vector is annihilated by all
operators from $\fn\oplus(\fn_-)^a$. 
\end{lem}
\begin{proof}
Let us start with the case $b=0$. In this case the cyclic vector
is given by $w_n\wedge \dots\wedge w_{n-k+1}$ and the cocylic vector is 
$w_1\wedge \dots\wedge w_{k}$. As mentioned above (see Example \ref{ex:b=0}) for $b=0$ 
the action of $\fI_1$ coincides with the action of $\fb$ (i.e. $(\fn_-)^a$ acts trivially). 
Clearly, the cocylic vector $cc^{(0)}_k$ is killed by both $\fn$ and $(\fn_-)^a$. 
Now the general case follows form the existence of the rotation symmetry from Lemma \ref{lem:cycsym}.
\end{proof}

\subsection{Family of modules}
Let $\la=\sum_{i=1}^{n-1} m_i\om_i$ be a dominant integral $\msl_n$ weight. We fix a decomposition
$\la=\sum_{b\in \Delta_0} \la^{(b)}$ into a sum of $n$ dominant integral weights (one for
each vertex of our cyclic quiver). 
In what follows we denote by $\bar\la$ the collection $(\la^{(b)})_{b\in \Delta_0}$.
Each dominant integral weight $\la$  gives rise to an irreducible highest weight $\msl_n$ module $V_{\la}$
with lowest weight vector $v_\la$;
we naturally extend the $\msl_n$ action to the $\mgl_n$ action (considering $\la$ as a partition with the 
last zero entry). 

\begin{rem}
We  use the lowest weight vectors instead of the highest weight vectors in order to make the picture compatible
with the theory of affine Demazure modules and Iwahori algebras. However, this choice is not important due
to the Weyl group symmetry.      
\end{rem}

The Cartan embedding 
$V_\la\hookrightarrow \bigotimes_{b\in \Delta_0} V_{\la^{(b)}}$, $v_\la\mapsto \otimes_{b\in \Delta_0} v_{\la^{(b)}}$ realizes $V_\la$ inside the tensor product; by definition,  
$V_\la=\U(\msl_n)\cdot \otimes_{b\in \Delta_0} v_{\la^{(b)}}=\bigodot_b V_{\la^{(b)}}$.
The tensor product admits a natural structure of $\fI_1(\veps)$-module for any $\veps$ induced
via the embedding $\fI_1(\veps)\hookrightarrow \bigoplus_{b\in \Delta_0} \mgl_n$.
For $\veps\ne 0$ we define
\[
V_{\bar\la}(\veps)=\U(\fI_1(\veps))\cdot \otimes_{b\in \Delta_0} v_{\la^{(b)}}\subset \bigotimes_{b\in \Delta_0} V_{\la^{(b)}}.
\]

\begin{lem}
For $\veps\ne 0$ one has an isomorphism of representations 
$V_\la\simeq  V_\barla(\veps)$ 	with respect to the identification $\mgl_n\simeq \fI_1(\veps)$.
\end{lem}
\begin{proof}
Recall the notion of $b$-admissibility (Definition \ref{def:badm}). 
We show that for any $b=1,\dots,n-1$ any irreducible module $V_\mu$ is $b$-admissible. 
For this we need to find an isomorphism $f:V_\mu\to V_\mu$ such that 
$f(v_\mu)=v_\mu$ and $\mathrm{sh}_b(x)\circ f = f\circ x$ for any $x\in\fn$ (as operators on $V_\mu$).
The existence of such an isomorphism (in a more general settings) is shown in Lemma \ref{lem:badm} 
\end{proof}

For $\veps=0$ we set 
\begin{equation}
V_\barla(0)=\lim_{\veps\to 0} V_\la(\veps) \subset \bigotimes_{b\in \Delta_0} V_{\la^{(b)}}.
\end{equation}
The spaces $V_\barla=V_\barla(0)$ form a special case of the modules $\bfD(0)$ 
from the previous section.

\begin{rem}
One has the following explicit construction.
Let us consider the algebra 
$\fI_1(\veps)$ inside the direct sum $\bigoplus_{b\in \Delta_0} \mgl_n\T_\bC \bC[\veps]$.
In other words,
let us treat the parameter $\veps$ as an auxiliary variable. This algebra acts on the space
$\bigotimes_{b\in \Delta_0} V_{\la^{(b)}}[\veps]$ and $V_\barla(\veps)$ is a subspace generated by
$\fI(\veps)$ from the tensor product of cyclic vectors. One gets the standard decreasing
filtration by the spaces $F_r$, $r\ge 0$ given by
\[
F_r=V_\barla(\veps)\cap \veps^r \bigotimes_{b\in \Delta_0} V_{\la^{(b)}}[\veps].  
\]     
Then $V_\barla(0)=\bigoplus_{r\ge 0} \left(\veps^{-r} F_r\right) |_{\veps=0}$ (we note that the sum if effectively
finite).
\end{rem}

\begin{rem}
It is important to note that $V_\barla$ in general strictly contains 
$\U(\fI_1).\otimes_{b} v_{\la^{(b)}}$, i.e. $V_\barla$ is not generated from the tensor product 
of cyclic vectors.  
\end{rem}

\begin{lem}\label{lem:Iwaction}
The space $V_\barla(0)$ as a subspace of $\bigotimes_b V_{\la^{(b)}}$ admits a natural action of the 
algebra $\fI_1$ compatible with the embedding $\fI_1\subset \bigoplus_b \mgl_n$. 
\end{lem}
\begin{proof}
As in Lemma \ref{lem:D0Iw}, the algebra $\fI_1$ is the limit $\veps\to 0$ of the algebras $\fI_1(\veps)$. 
\end{proof}

\subsection{Algebraic coherence conjecture}

Recall the Weyl group $W=S_n$. For a dominant integral weight $\la$ and $\sigma\in W$ let 
$v_{\sigma\la}\in V_\la$ be an extremal weight vector of weight $\sigma\la$. In  particular, if 
$\sigma=e$, then $v_{\sigma\la}$ is a highest weight vector and if $\sigma=w_0$ is the longest
element, then $v_{\sigma\la}=v_\la$ is the lowest weight vector.
Let $V_\barla=V_\barla(0)$.

\begin{lem}\label{lem:extrbel}
For any $\sigma\in S_n$ the tensor product of extremal vectors 
$\otimes_b v_{\sigma {\la^{(b)}}}$ belongs to $V_\barla$.
\end{lem}
\begin{proof}
We note that $v_{\sigma\barla}$ belongs to $V_\barla(\veps)$ for any non-zero $\veps$.
In fact, any extremal vector $v_{\sigma\la}$ can be reached from $v_\la$ by successive
application of root vectors, which are still present in $\fI_1(\veps)$ for any non-zero
$\veps$. Now it remains to note that  $v_{\sigma\barla} \in \bigotimes_b V_{\la^{(b)}}$
spans the weight $\sum_b \sigma\la^{(b)}$ subspace. 
\end{proof}	
	
\begin{conj}\label{conj:main}
The space $V_\barla$ is generated from the vectors $v_{\sigma\barla}$ by the action of $\fI_1$, i.e.
\[
V_\barla = \sum_{\sigma\in W} \U(\fI_1) v_{\sigma\barla}\subset \bigotimes_b V_{\la^{(b)}}.
\]
\end{conj}	

Lemmas \ref{lem:Iwaction} and \ref{lem:extrbel} imply that 
$V_\barla \supset \sum_{\sigma\in W} \U(\fI_1) v_{\sigma\barla}$. 
Since $\dim V_\barla = \dim V_{\sum_b \la^{(b)}}$, 
Conjecture \ref{conj:main} is equivalent to the equality 
\begin{equation}\label{eq:numconj}
\dim \sum_{\sigma\in W} \U(\fI_1) v_{\sigma\barla} = \dim V_{\sum_b \la^{(b)}}.
\end{equation}

\begin{rem}
The computer experiments (see Appendix \ref{sec:app})  support the numerical 
conjecture \eqref{eq:numconj}.
\end{rem}

\begin{thm}\label{thm:fund}
Let us fix $k=1,\dots,n-1$ and assume that all weights $\la^{(b)}$ are proportional 
to $\omega_k$. Then Conjecture \ref{conj:main} holds true and each subspace 
$\U(\fI_1) v_{\sigma\barla}$ is isomorphic to an affine Demazure module. 
\end{thm}	
\begin{proof}
The claim follows from the Zhu theorem \cite{Zhu14}, proving the Pappas--Rapoport coherence conjecture.
We provide the details in a more general case of Demazure modules, see Theorem \ref{thm:cohconj}.
\end{proof}

\begin{cor}
The $\fI_1$ action on the space $V_\barla$ agrees with the Iwahori algebra 
action on the corresponding sum of affine Demazure modules (see Corollary 
\ref{cor:Dprop} for more details).
\end{cor}

\subsection{Embeddings}
Let $\la^{(b)}=\sum_{i=1}^{n-1} m^{(b)}_i\om_i$, $\barla=(\la^{(b)})_b$. We consider the 
collections $\barla(i)$, $i=1,\dots,n-1$, formed by weights which are 
multiples of $\om_i$. More precisely, 
\[
\barla(i) = (m^{(0)}_i\om_i,m^{(1)}_i\om_i,\dots, m^{(n-1)}_i\om_i).
\]
Then Theorem \ref{thm:fund} tells us that $V_{\barla(i)}$ is generated from the  extremal weight vectors.  

\begin{prop}\label{prop:KK}
Let $\barla$ be a collection such that Conjecture \ref{conj:main} holds true. Then
$V_{\barla}$ is embedded into the tensor product $\bigotimes_{i=1}^{n-1} V_{\barla(i)}$ 
and the image of the embedding is generated by the $\fI$ action from the vectors
$\bigotimes_{i=1}^{n-1} v_{\sigma\barla(i)}$ for all $\sigma\in S_n$.
\end{prop}
\begin{proof}
We consider the chain of embeddings
\[
V_{\barla} \hookrightarrow \bigotimes_{b=0}^{n-1} V_{\la^{(b)}} \hookrightarrow
\bigotimes_{b=0}^{n-1} \bigotimes_{i=1}^{n-1} V_{m^{(b)}_i\om_i} \simeq 
\bigotimes_{i=1}^{n-1} \left(\bigotimes_{b=0}^{n-1}  V_{m^{(b)}_i\om_i}\right),
\]
where the second embedding is obtained by taking the Cartan components.
Then $V_{\barla(i)}$ sits inside $\bigotimes_{b=0}^{n-1}  V_{m^{(b)}_i\om_i}$
and the extremal weight vectors $v_{\sigma\barla}$ are the tensor products of the 
extremal weight vectors $v_{\sigma\barla^{(i)}}$.
\end{proof}

\begin{cor}\label{cor:KK}
Let $\barla$ be a collection such that Conjecture \ref{conj:main} holds true.
Then $V_{\barla}$ is isomorphic as an $\fI$ module to a sum of affine Kostant--Kumar
modules inside a tensor product of integrable highest weight representations. 
\end{cor}
\begin{proof}
We recall the definition of the Kostant--Kumar modules in the affine settings \cite{Kum88,Kum89,KRV24}. Let 
$\overline w=(w_1,\dots,w_r)$ be a collection of affine Weyl group elements and let 
$\overline \Lambda=(\Lambda^{(1)},\dots,\Lambda^{(r)})$ be a collection of affine dominant
integral weights. Then the Kostant--Kumar module $K(\overline w, \overline \Lambda)$ is a
cyclic $\fI$ module defined by
\[
K(\overline w; \overline \Lambda) = \U(\fI). \bigotimes_{i=1}^r v_{w_i}(\Lambda^{(i)})\subset 
\bigotimes_{i=1}^r D_{w_i}(\Lambda^{(i)}), 
\] 
where $v_w(\Lambda)\in L(\Lambda)$ is the extremal vector of weight $w(\Lambda)$.
Recall the embedding $V_{\barla}\subset \bigotimes_{i=1}^{n-1} V_{\barla(i)}$ above.
We know that each $\fI$ module $V_{\barla(i)}$ sits inside the representation $L(\sum m_i^{(b)}\Lambda_b)$ as a sum of Demazure modules 
\[
V_{\barla(i)} \simeq \sum_{\sigma\in W} D_{z^{\sigma\om_i}}(\sum_b m_i^{(b)}\Lambda_b).
\]
Hence Proposition \ref{prop:KK} implies that $V_{\barla}$ is isomorphic to the sum over 
all $\sigma\in W$ of Kostant--Kumar
modules $K(\overline w(\sigma); \overline \Lambda(\sigma))$ with $r=n-1$, where
\[
\overline w(\sigma)= (z^{\sigma\om_1},\dots,z^{\sigma\om_{n-1}}),\
\overline \Lambda(\sigma) = (\sum_b m_1^{(b)}\Lambda_b,\dots,\sum_b m_{n-1}^{(b)}\Lambda_b).
\] 
\end{proof}

\begin{rem}
Recall \cite{Go01,He13,Pa18} that in the classical situation (with all weights $\la^{(b)}$ being multiples of one fundamental weight $\omega$) the degeneration from $V_{k\omega}$ to $V_\barla$ can be understood
geometrically via the global Schubert varieties, providing a degeneration of the spherical Schubert 
varieties inside affine Grassmannians to the union of the Schubert varieties inside flag varieties.
This degeneration can be described in terms of quiver Grassmannians (see \cite{FLP23,FLP24}). 
Hence it is natural to expect that the more general case we consider in this
paper (mixing different fundamental weights) has to do with the quiver flag varieties.
More precisely, we expect that the general $\fI$ modules $V_\barla$ can be realized as spaces of
sections of natural line bundles on certain subvarieties of the quiver flag varieties. The subvarieties
in question are the Iwahori group closures of the products of the torus fixed points contained
in the quiver flag varieties.   
\end{rem}

\section{Demazure modules} \label{sec:Dem}
In this section we study the general construction of the modules $\bfD(0)$ for the case
of spherical affine Demazure modules, generalizing the discussion from the previous section. 

\begin{lem}\label{lem:badm}
Let $D$ be a $\msl_n(\bO)$ stable affine Demazure module. Then $D$ is $b$-admissible 
for any $b$. 
\end{lem}
\begin{proof}
Let $v\in D$ be the lowest weight cyclic vector, so the current algebra $\fn(\bO)$ 
(even $\fn[z]$) generates $D$
from $v$. Recall that the $\bO$ linear map $\mathrm{sh}_b$ preserves  $e_az^k$ for 
$a\ne b$ and sends $e_bz^k$ to $e_bz^k(z+\veps)$. We extend $\mathrm{sh}_b$ to the
$\bO$ linear automorphism of the Iwahori algebra by setting 
$\mathrm{sh}_b f_b = f_b(z+\veps)^{-1}$ and $\mathrm{sh}_b f_a = f_a$ for $a\ne b$ 
(recall that $\veps\ne 0$). Since the operators $e_az^k$, $a\ne b$ and $e_bz^k(z+\veps)$ 
generate $D$ from $v$ ($D$ is graded and finite-dimensional, hence there exists an $N$ 
such that all elements of the form $xz^N$ act trivially), it suffices to check that
all the defining relations of $D$ as a $\fI$ module still hold true after the twist by
$\mathrm{sh}_b$.
Recall (see \cite{J85,M88}) that the defining relation of $D$ are of the form 
$\fn_-(\bO)v=0$, $z\fh(\bO)v=0$ and $(e_\al z^k)^{N_\al,k}v=0$ for positive roots 
$\al$, $k\ge 0$ and certain numbers $N_\al,k$ 
(with an assumption that $v$ has a prescribed weight).
One easily sees that the $\fI$ automorphisms  $\mathrm{sh}_b$ preserve these relations.
\end{proof}

Let $D_b$, $b=0,\dots,n-1$ be $\msl_n$-invariant affine Demazure modules. Let us fix
the lowest weight vectors $v_b\in D_b$.
The subspace generated by the $\msl_n$ action on $v_b$ is isomorphic to the highest weight
module $V_{\la^{(b)}}$. For a Weyl group element $\sigma\in S_n$ let 
$v_{\sigma\la^{(b)}}\in V_{\la^{(b)}}\subset D_b$ be the extremal vector of the  weight $\sigma\la^{(b)}$.

\begin{cor}\label{cor:D(0)Dem}
For any $\sigma\in S_n$ the space $\bfD(0)$ contains the subspace 
$\U(\fI).\bigotimes_{b\in \Delta_0}v_{\sigma\la^{(b)}}$.
\end{cor}
\begin{proof}
The space $\bfD(0)$ is a weight subspace of the tensor product 
$\bigotimes_{b\in \Delta_0} D_b$.
Recall that $z\fh[z]v_b=0$, hence $z\fh[z]v_{\sigma\la^{(b)}}=0$ for any $\sigma$. 
We conclude that the $\sigma(\la^{(b)})$-weight subspace  of $D_b$ is one-dimensional. 
Clearly, $\bigotimes_{b\in \Delta_0}v_{\sigma\la^{(b)}}\in \odot_{b} V_{\la_b}$ (this
is the extremal weight vector in the $V_{\sum \la^{(b)}}$). This implies the desired result.
\end{proof}

In what follows we denote $\bigotimes_{b\in \Delta_0}v_{\sigma\la^{(b)}}$ by $d_{\sigma}$.
In particular, the vector $d_{w_0}$ is
the tensor product of the lowest weight vectors of the representations $V_{\la^{(b)}}\subset D_b$.
Our goal is to describe $\bfD(0)$ as a module over the Iwahori algebra $\fI$. 

\subsection{Multiples of a level one weight}\label{subsec:mult}
Let us fix a dominant integral weight $\la$ and a number $k\in\bZ_{\ge 0}$. Let $D_\la\subset L(\Lambda_j)$ be the  Demazure module
inside integrable level one irreducible representations of affine $\msl_n$
such that the weight of the highest weight vector of $D_\la$ is equal to $\la$.
Let $D_{k,\la}$ be the level $k$ Demazure module defined as
$D_\la^{\odot k}$. We fix a composition $\bk=(k_b)_{b=0}^{n-1}$ such that  $|\bk|=\sum_{b=0}^{n-1} k_i=k$
and consider the Demazure modules $D_b=D_{k_b,\la}$. 
The general construction above defines the space $\bfD(0)$ inside 
$\bigotimes_{b\in \Delta_0} D_{k_b,\la}$. We denote this special fiber $\bfD(0)$ by $D_{\bk,\la}$.

Let $\Lambda_\bk=\sum_{b\in\Delta_0} k_b \Lambda_b$; in  particular, the level of 
 $\Lambda_\bk$ is equal to $k$. Recall the subspaces 
 $\U(\fI).\bigotimes_{b\in \Delta_0}v_{\sigma\la^{(b)}}$ inside $D_{\bk,\la}$ 
 from Corollary \ref{cor:D(0)Dem}, where the Iwahori algebra acts via the isomorphism $\fI\simeq \fa(0)$ (see Lemma \ref{lem:aIw}); in our case $\la^{(b)}=k_b\la$.
 
\begin{prop}\label{prop:submod}
The sum over $\sigma\in W$ of Demazure submodules of $L(\Lambda_\bk)$,
corresponding to the affine Weyl group elements $z^{\sigma\la}$,
admits an $\fI$ equivariant embedding into $D_{\bk,\la}$. 
The image of the embedding is equal to the sum over $\sigma\in W$ of subspaces  $\U(\fa(0)).\bigotimes_{b\in \Delta_0}v_{\sigma\la^{(b)}}$.
\end{prop}
\begin{proof}
We need an embedding of the sum of Demazure submodules 
of $L(\Lambda_\bk)$ into the tensor product $\bigotimes_{b\in\Delta_0} D_{k_b,\la}$, which inrewines the $\fI$ and the $\fa(0)$ actions and 
sends the cyclic vector of $D_{z^{\sigma\la}}(\Lambda_\bk)$ to $\bigotimes_{b\in \Delta_0}v_{\sigma\la^{(b)}}$. Our first goal is to construct such an embedding 
for each module $D_{z^{\sigma\la}}(\Lambda_\bk)$.

Since for any two affine dominant integrable weights $\Lambda^1$ and $\Lambda^2$
and an affine Weyl group element $\tau$ the Demazure module $D_\tau(\Lambda^1+\Lambda^2)$
is the Cartan component inside   $D_\tau(\Lambda^1)\T D_\tau(\Lambda^2)$, it is enough 
to show that $D_{z^{\sigma\la}}(\Lambda_b)$ embeds into $D_\la$ for any $b\in\Delta_0$
and a level one Demazure submodule $D_\la\subset L(\Lambda_j)$. 
More precisely, we need to show that 
there exists an embedding  $D_{z^{\sigma\la}}(\Lambda_b)\subset D_\la$ sending the
cyclic vector to $v_{\sigma\la}$ and  interwining the standard $\fI$ action on 
$D_{z^{\sigma\la}}(\Lambda_b)$ and the $\fa(0)$ action on $D_\la$ coming from the projection 
of formulas from Section \ref{sec:construction} to the $b$-th component (see Lemma \ref{lem:aIw}).

Recall (see e.g. \cite{FoL06}) that for each integral $\msl_n$ weight $\mu$ there exists 
a (unique) affine level one weight $\Lambda$  such that $L(\Lambda)$ contains a Demazure
submodule $D(\mu)$, whose cyclic vector has weight $\mu$ (in particular, $D_\la=D(w_0\la)$
for the longest element $w_0\in W$). Our goal is to embed 
$D_{z^{\sigma\la}}(\Lambda_b)=D(\om_b+\sigma\la)$ into $D_\la$ sending the (cyclic) 
weight $\om_b+\sigma\la$ vector to $v_{\sigma\la}$. 

The cyclic symmetry of the Dynkin diagram of the affine Lie algebra $\widehat\msl_n$
induces an isomorphisms between the level one integrable irreducible representations 
$L(\Lambda_b)$, $b\in\Delta_0$ (of course, these are only the isomorphisms of vector spaces).
Let us examine how the action of the Iwahori algebra  transforms under the isomorphism between $\Lambda_{b+j}$ and $\Lambda_j$. The algebra $\fI$ is generated by $E_{1,2},\dots,E_{n-1,n}$
and $zE_{n,1}$. The $b$-step Dynkin diagram rotation changes the generators to 
\[
E_{1,2},\dots,E_{b-1,b},zE_{b,b+1},E_{b+1,b+2},\dots, E_{n-1,n}, E_{n,1}.
\]
Let us denote the Lie algebra generated by these elements by $\fI^{(b)}\subset \msl_n(\eO)$;
in particular, $\fI^{(0)}=\fI$.
Then $\fI^{(b)}$ is exactly the projection of $\fa(0)$ to the $b$-th component.
Since the (vector space) isomorphism $\Lambda_{b+j}\simeq \Lambda_j$ shifts the weights by $\om_b$ and, hence, sends the cyclic vector of $D(\om_b+\sigma\la)$ to $v_{\sigma\la}$, we obtain the 
isomorphism $D_{z^{\sigma\la}}(\Lambda_b)\simeq \U(\fI^{(b)})v_{\sigma\la}$.
	
In order to complete the proof of our Proposition, we need to show that the intersections 
between $D_{z^{\sigma\la}}(\Lambda_\bk)$ inside $L(\Lambda_\bk)$ agree with that of 
$\U(\fa(0)).v_{\sigma\la}$ inside $D_{\bk,\la}$. To this end, we recall (see \cite{Ka21,Kum02}) 
that an intersection of Demazure nodules is equal to the sum of (smaller) Demazure
submodules contained in all the initial ones. Now using the same approach (i.e. the
identification of subspaces in different integrable representations), one arrives at
the desired claim. 	
\end{proof}

\begin{thm}[Type $A$ coherence conjecture]\label{thm:cohconj}
There exists $\fI$ equivariant embedding 
\[
D_{\bk,\la} \hookrightarrow L(\Lambda_\bk), \ \Lambda_\bk=\sum_{b\in\Delta_0} k_b \Lambda_b.
\]
The image of the embedding is equal to the sum of Demazure submodules $D_{z^{\sigma\la}}(\Lambda_\bk)$. 
\end{thm} 
\begin{proof}
We use the type $A$ case of the Zhu theorem (which proves the Pappas--Rapoport coherence
conjecture) and show that the special case of our construction (as above, 
$D_b = D_{k_b,\la}$) is the algebraic side of the geometric degeneration.

Recall the embedding $D_{k_b,\la}\subset L(k_b\Lambda_j)$ of the affine Demazure 
module into the irreducible integrable $\widehat{\msl}_n$ module of highest weight
$k_b\Lambda_j$ (for some $j\in \Delta_0$). Let $\aGr_j$ be the $j$-th affine
Grassmannian. In particular, there is an embedding $\aGr_j\subset\bP(L(\Lambda_j))$ and the space
of sections of $\eO(1)$ on $\aGr_j$ is isomorphic to the dual of $L(\Lambda_j)$.

The Demazure module $D_{k_b,\la}$ is generated by the action of the universal enveloping 
algebra of the current algebra $\msl_n[z]$ on the extremal vector $d_{k_b,\la}\in L(\Lambda_j)$.
Let $X_{\la}\subset \aGr_j$ be the corresponding spherical Schubert variety (it does not depend on $k_b$). 
By definition, for each $k_b>0$ one gets a $SL_n[z]$-equivariant embedding  
$X_{\la}\hookrightarrow \bP(D_{k_b,\la})$. 
Slightly abusing notation, we denote the restriction of $\eO(1)$ to $X_\la$ by the same symbol.
One has $H^0(X_\la,\eO(k))\simeq D_{k,\la}^*$.

We consider the global affine Grassmannain $\bf{Gr}$ inside $\bC\times \prod_{\bZ/n\bZ}\aGr_j$ (see section \ref{subsec:globalaffine}).
The general fiber of the projection map $\pi: \bf{Gr}\to\bC$ (projection to the first coordinate $\veps$)
is isomorphic to the affine Grassmannian $\aGr_j$ and the special fiber (over $\veps=0$)
is isomorphic to the affine flag variety ${\mathcal Fl}$. Now there exists a subfamily
${\bf X}_\la$ whose general fiber is isomorphic to the spherical Schubert variety 
$X_\la\subset \aGr_j$ and the special fiber is a union of Schubert varieties inside the affine
flag variety. Let us consider a line bundle on ${\bf X}_\la$ which restricts to the
general fiber as $\eO(|\bk|)$ and to the special fiber as  $\eL(\Lambda_\bk)$ (see \cite{Zhu14,HY24}). 
The space of sections $H^0(X_\la,\eO(|\bk|)$ is identified with the dual Demazure module
inside $L(\Lambda_\bk)$ whose highest weight is $|\bk|\la$, i.e.
\[
H^0(X_\la,\eO(|\bk|)^*\simeq \bigodot_{b=0}^{n-1} D_{k_b,\la}.
\] 
The dual space of section of $\eL(\Lambda_\bk)$ on the special fiber is identified with 
the sum of Demazure submodules inside $L(\Lambda_\bk)$, corresponding to  the Weyl group elements $z^{\sigma\la}$, $\sigma\in W$. In particular, 
as the Iwahori algebra module it is generated from the extremal weight vectors.
By the Zhu theorem the dimensions of the spaces of sections on the special and general fibers coincide. 

To finalize the proof, we note that by Proposition \ref{prop:submod} the module
$D_{\bk,\la}$ contains the sum of Demazure submodules $D_{z^{\sigma\la}}(\Lambda_\bk)$.
However, $\dim D_{\bk,\la}$ is equal to $\dim \bigodot_{b=0}^{n-1} D_{k_b,\la}$ and hence by the Zhu theorem 
\[
\dim D_{\bk,\la} = \dim \sum_{\sigma\in W} D_{z^{\sigma\la}}(\Lambda_\bk),
\]
which implies that the above containment is the equality.
\end{proof}
  
\begin{cor}\label{cor:Dprop}
The module $D_{\bk,\la}$ admits the following properties:
\begin{itemize}
\item $D_{\bk,\la}$ is generated by the action of $\fI$ from vectors
$d_\sigma\in \bigotimes D_{k_b,\la}$, $\sigma\in S_n$,
\item each space $\U(\fI)d_\sigma$ is isomorphic to an affine Demazure module in  $L(\sum_{b} k_b \Lambda_b)$,
\item  $D_{\bk,\la}$ is $\fI$ cocyclic,
\item the cocyclic vector of $D_{\bk,\la}$ is the highest weight vector of $L(\sum_{b} k_b \Lambda_b)$.
\end{itemize} 
\end{cor}

\subsection{Demazure generalization and Kostant--Kumar modules}\label{subsec:Demgen}
Let us consider a more general class of collections of Demazure modules. Let $w$ be an element of the affine Weyl group and let 
$\overline{\Lambda}=(\Lambda^{(0)},\dots\Lambda^{(n-1)})$ be a collection of affine integral dominant weights.
We consider the Demazure modules $D_w(\Lambda^{(b)})$ and assume that all these Demazure 
modules are spherical (i.e. $\msl_n[z]$-invariant). The Cartan component inside the tensor 
product of these Demazure modules is described as 
\[
\bigodot_{b=0}^{n-1} D_w(\Lambda^{(b)})\simeq D_w(\sum_b\Lambda^{(b)}).
\]
Let $\bfD(0)$ be the corresponding degeneration (a module of the Iwahori algebra) of $D_w(\sum_b\Lambda^{(b)})$.  
It is natural to ask when does the algebraic coherence conjecture holds true, i.e. when is $\bfD(0)$
generated by the action of the Iwahori algebra from the set of extremal weight
vectors. We provide some examples and non-examples in section \ref{sec:examples}. 

\begin{example}
Let $w=w_0\in W\subset W^a$ be the longest element in the finite Weyl group. Then $D_w(\Lambda^{(b)})$
is isomorphic to a finite-dimensional module $V_{\Lambda^{(b)}_{fin}}$, where $\Lambda^{(b)}_{fin}$
is the finite part of the affine weight. Hence we are in the situation of Conjecture \ref{conj:main}.  
\end{example}

\begin{example}
Let us fix an affine fundamental weight $\Lambda_j$ and an affine Weyl group element $w$ such that 
$D_\la$ (as in subsection \ref{subsec:mult}) is equal to $D_w(\Lambda_j)$. For a composition 
$\bk=(k_b)_{b=0}^{n-1}$ let $\Lambda^{(b)}=k_b\Lambda_j$. Then $\bfD(0)$ is isomorphic to
$D_{\bk,\la}$ as in Theorem \ref{thm:cohconj}.  
\end{example}

Now let us explain the connection between the modules $\bfD(0)$ as above and the Kostant--Kumar
modules. Let us fix an affine Weyl group element $w$ and the affine weights $\Lambda^{(b)}$.
Let us decompose each weight $\Lambda^{(b)}=\sum_{j=0}^{n-1} k_b^{(j)}\Lambda_j$ as a linear
combination of affine fundamental weights. Then one has an embedding 
\begin{equation}\label{eq:tpemb}
D_w(\Lambda^{(b)})\subset \bigotimes_{j=0}^{n-1} D_w(\Lambda_j)^{\T k_b^{(j)}}.
\end{equation}
Let us define the finite weights $\la(j)$ such that $D_{\la(j)}=D_w(\Lambda_j)$. Then \eqref{eq:tpemb}
induces the embedding $\bfD(0)\subset \bigotimes_{j=0}^{n-1} D_{\bk(j),\la(j)}$, where 
$\bk(j)= (k_b^{(j)})_b$. Now similarly to Corollary \ref{cor:KK} (assuming algebraic coherence conjecture holds true), one gets the isomorphism 
$\bfD(0)\simeq \sum_{\sigma\in W} K(\bar w(\sigma),\bar \Lambda)$, 
where
\[
\bar w(\sigma) = (z^{\sigma\la(0)},\dots,z^{\sigma\la(n-1)}),\
\bar \Lambda = (\sum_{b=0}^{n-1} k_b^{(0)}\Lambda_b,\dots,\sum_{b=0}^{n-1} k_b^{(n-1)}\Lambda_b).
\]

\begin{rem}
As mentioned above, the central geometric objects showing up in \cite{Zhu14} and \cite{HY24} 
are the (global) Schubert varieties. The geometric objects responsible for the more general algebraic
setup are the Kostant--Kumar analogues of the Schubert varieties. More precisely, a Kostant--Kumar
module is realized inside the tensor products of Demazure modules as $\U(\fI)$ span of
the tensor products of cyclic vectors. Similarly, a Kostant--Kumar Schubert variety is defined
as the closures of the $\bI$ orbits of the cyclic lines inside the products of Schubert varieties.
We expect that the modules $\bfD(0)$ considered above can be realized as spaces of sections of
natural line bundles on (unions of) Kostant--Kumar Schubert varieties. 
\end{rem}

\section{Examples}\label{sec:examples}
In this section we provide several explicit low rank examples of the general construction.
The irreducible $\msl_n$ modules and the affine Demazure modules considered
below are generated from the highest weight vectors (as opposed to the lowest weight vectors in the main body of the paper). 
This choice allows simpler visualization of the examples.
  
\subsection{Rank one case}
We start with the case of $\fg=\msl_2$, we use the notation from section \ref{sec:fin-dim}.
Let $e,h,f$ be the standard basis of $\msl_2$. 
For two non-negative integers $\la_0,\la_1\ge 0$ we consider the tensor product 
$V_{\la_0\omega}\T V_{\la_1\omega}$ of two irreducible representations 
of $\msl_2$ 
with 
highest weight vectors $v_{\la_0}$ and $v_{\la_1}$. The Weyl group
consists of two elements $\mathrm{id}$ and $s$; we denote by $v_{s\la}\in V_
\la$ the lowest weight (extremal) vector. 
The operator $f(\veps) =  f\T \mathrm{id} + \veps \mathrm{id}\T f$ being applied multiple times to $v_{\la_0}\T v_{\la_1}$ produces the following vectors 
\begin{gather*}
f(\veps)^m	v_{\la_0}\T v_{\la_1} = 
\sum_{i=0}^m \veps^{m-i} f^{i}v_{\la_0}\T f^{m-i}v_{\la_1},\ 0\le m\le \la_0,\\
f(\veps)^{\la_0+m}	v_{\la_0}\T v_{\la_1} =
\veps^m \sum_{i=0}^{\la_0} \veps^{\la_0-i} f^iv_{\la_0}\T f^{\la_0+m-i}v_{\la_1}, 1\le m\le \la_1. 
\end{gather*}
The lowest $\veps$-degree parts of the above vectors are equal to 
\begin{gather}
\label{f}	
(f(\veps)^m	v_{\la_0}\T v_{\la_1})^o = f^mv_{\la_0}\T v_{\la_1},\ 0\le m\le \la_0,\\
\label{e}
(f(\veps)^{\la_0+m}	v_{\la_0}\T v_{\la_1})^o =
f^{\la_0}v_{\la_0}\T f^{m}v_{\la_1}, 1\le m\le \la_1. 
\end{gather}
Vectors \eqref{f} and \eqref{e} span the space $V_\barla$.
The Lie algebra $\fI_1\subset \mgl_2\oplus\mgl_2$  is the span of three
elements $(f,0)$, $(0,e)$ and $(h,h)$. 
One sees explicitly that 
\begin{gather*}
\U(\fI_1)(v_{\la_0}\T v_{\la_1}) = \mathrm{span}\{f^mv_{\la_0}\T v_{\la_1},\ 0\le m\le \la_0\},\\
\U(\fI_1)(v_{s\la_0}\T v_{s\la_1}) =  \mathrm{span}\{f^{\la_0}v_{\la_0}\T f^{m}v_{\la_1}, 0\le m\le \la_1\},
\end{gather*}
and the intersection of these two spaces is the span of the vector 
$v_{s\la_0}\T v_{\la_1}$
(we note that $f^{\la_0}v_{\la_0}$ is proportional to $v_{s\la_0}$ and
$f^{m}v_{\la_1}$ is proportional to $e^{\la_1-m}v_{s\la_1}$).
Hence, $V_{\barla}$ is generated from two vectors $v_{\la_0}\T v_{\la_1}$ and $v_{s\la_0}\T v_{s\la_1}$ by the action of the Lie algebra $\fI_1$. We also see that $V_\barla$ is cocylic with the cocyclic vector   $v_{s\la_0}\T v_{\la_1}$.

Let $\Lambda_0$ and $\Lambda_1$ be integrable level one irreducible representations of 
$\widehat{\msl}_2$. There is an embedding $V_\barla\hookrightarrow L(\Lambda)$ with
$\Lambda= \la_0\Lambda_1+\la_1\Lambda_0$. Let $u_\Lambda\in L(\Lambda)$ be the highest weight vector and let $s_0,s_1$ be the standard generators of the affine Weyl group of type $A_1^{(1)}$.
In particular, $L(\Lambda)$ contains the extremal weight vectors $u_{s_0\Lambda}$ and $u_{s_1\Lambda}$.
Then the embedding  sends 
\[
v_{s\la_0}\T v_{\la_1}\mapsto u_\Lambda,\     
v_{\la_0}\T v_{\la_1}\mapsto u_{s_0\Lambda},\
v_{s\la_0}\T v_{s\la_1}\mapsto u_{s_1\Lambda}.
\]     

Here is the picture for $\la_0=2$, $\la_1=3$.
We start with the tensor product
\[
\begin{tikzcd}[scale cd=1]
\bullet & \arrow[l,"f",swap]  \bullet & \arrow[l,"f",swap] v_2\quad \T \quad 
\bullet & \arrow[l,"f",swap]\bullet &\arrow[l,"f",swap] \bullet  & \arrow[l,"f",swap] v_3,
\end{tikzcd}	
\]
consider inside the Cartan component
\[
\begin{tikzcd}
\bullet& \arrow[l,"f",swap]\bullet  & \arrow[l,"f",swap]\bullet &
\arrow[l,"f",swap] \bullet & \arrow[l,"f",swap] \bullet & \arrow[l,"f",swap] v_2\T v_3
\end{tikzcd}	
\]
and then deform it to the Iwahori algebra module
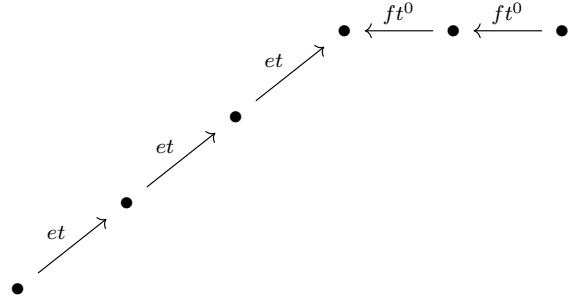
\[
\begin{tikzcd}
& & &\bullet & \arrow[l,"ft^0",swap]\bullet  &\arrow[l,"ft^0",swap] \bullet \\
& & \bullet\arrow[ur,"et"]   \\
& \bullet\arrow[ur,"et"] \\
\bullet\arrow[ur,"et"] 
\end{tikzcd}	
\]
which is a subspace in the irreducible integrable $\widehat{\msl}_2$ module 
$L(3\Lambda_0+2\Lambda_1)$ (this subspace a sum of two Demazure modules corresponding to the length
one elements $s_0$ and $s_1$ in the affine Weyl group).

We note that the same Cartan component $V_{5\omega}$
\[
\begin{tikzcd}
	\bullet& \arrow[l,"f",swap]\bullet  & \arrow[l,"f",swap]\bullet &
	\arrow[l,"f",swap] \bullet & \arrow[l,"f",swap] \bullet & \arrow[l,"f",swap] v_2\T v_3
\end{tikzcd}	
\]
sits inside another tensor product
\[
\begin{tikzcd}[scale cd=1]
	\bullet & \arrow[l,"f",swap]  \bullet & \arrow[l,"f",swap]  \bullet & \arrow[l,"f",swap]  \bullet & \arrow[l,"f",swap] v_2\quad \T \quad 
	\bullet & \arrow[l,"f",swap] v_3.
\end{tikzcd}	
\]
 One obtains another degeneration of $V_{5\omega}$ of the form
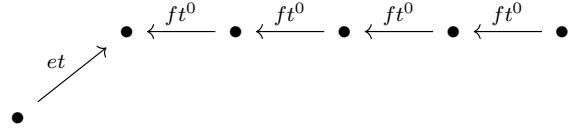
\[
\begin{tikzcd}
& \bullet & \arrow[l,swap,"ft^0"] \bullet & \arrow[l,swap,"ft^0"] \bullet 
& \arrow[l,swap,"ft^0"] \bullet &\arrow[l,swap,"ft^0"] \bullet \\
 \bullet \arrow[ur,"et"]
\end{tikzcd}	
\]
which is the sum of two Demazure modules inside the irreducible integrable $\widehat{\msl}_2$ module $L(\Lambda_0+4\Lambda_1)$.

\subsection{Rank two finite case}
Let us consider $\msl_3$ and $\la_0=\om_1$, $\la_1=\om_2$, $\la_2=\om_1$. This is one of the
simplest examples which does not show up in the context of the geometric coherence conjecture.
The $15$-dimensional Cartan component is isomorphic to $V_{2\om_1+\om_2}$; let us describe its
degeneration to the Iwahori module $V_\barla$.

Let $v_1,v_2,v_3$ and $v_{12}, v_{13}, v_{23}$ be the standard bases of $V_{\om_1}$ and
$V_{\om_2}$. Then the cocyclic vector in $V_\barla$ is $v_3\T v_{23}\T v_1$. We want to check 
that the action of $\fI_1$ (the Iwahori quotient) generates $15$-dimensional
subspace from the six extremal weight vectors inside $V_{\la_0}\T V_{\la_1}\T V_{\la_2}$.
The six operators we use are as follows:
\begin{gather*}
f_1t^0 = (f_1,f_1,0),\ f_2t^0 = (f_2,0,f_2),\ f_{12}t^0=(f_{12},0,0),\\
e_1t = (0,0,e_1),\ e_2t = (0,e_2,0),\  e_{12}t = (0,e_{12},e_{12}).
\end{gather*}
Here is the picture for $V_\barla$:
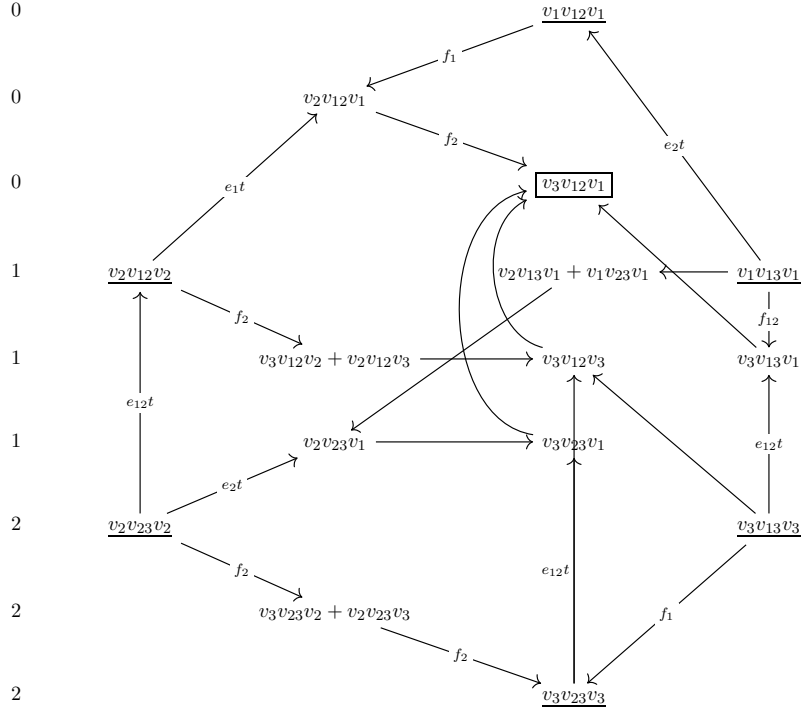
\[
\begin{tikzcd}[scale cd=0.7]
	0 & &  & \underline{v_1v_{12}v_1} \arrow[dl,"f_1"description] &   \\
	0 & & v_2v_{12}v_1 \arrow[dr,"f_2"description]&   &   \\
	0 & &  & \boxed{v_3v_{12}v_1}  &   \\
	1 & \underline{v_2v_{12}v_2} \arrow[uur,"e_1t"description] \arrow[dr,"f_2"description] &  & v_2v_{13}v_1 + v_1v_{23}v_1 \arrow[ddl]& \underline{v_1v_{13}v_1} \arrow[uuul,"e_2t"description]\arrow[l]\arrow[d,"f_{12}"description]\\
	1 & & v_3v_{12}v_2 + v_2v_{12}v_3\arrow[r]  & v_3v_{12}v_3 \arrow[uu,bend left=70]  &   v_3v_{13}v_1\arrow[uul]\\
	1 & & v_2v_{23}v_1 \arrow[r] &  v_3v_{23}v_1 \arrow[uuu,bend left=80]   &  \\
	2 & \underline{v_2v_{23}v_2}\arrow[ur,"e_2t"description]\arrow[uuu,"e_{12}t"description] \arrow[dr,"f_2"description]&  &  &      \underline{v_3v_{13}v_3}\arrow[ddl,"f_1"description]\arrow[uu,"e_{12}t"description]\arrow[uul]\\
	2 & &   v_3v_{23}v_2 + v_2v_{23}v_3 \arrow[dr,"f_2"description]&     & \\ 
	2 &  &  &    \underline{v_3v_{23}v_3} \arrow[uuu,"e_{12}t"]\arrow[uuuu] &\\
	&  
\end{tikzcd}	
\]
On the picture the extremal weight vectors are underlined, the cocyclic vector is boxed and 
the numbers on the left correspond to the standard $z$-degree of the Iwahori algebra (we note that 
the arrow from the lowest weight vector $v_3v_{23}v_3$ points to two different vectors, which means that
the result of the application of the operator $e_{12}z$ is a linear combination of the vectors 
$v_3v_{23}v_1$ and $v_3v_{12}v_3$).

\subsection{Rank two affine case}
In this subsection we provide several examples for the construction from section \ref{subsec:Demgen} for the $\widehat{\msl}_2$ Demazure modules.

Let $\Lambda^{(0)}=\Lambda^{(1)}=\Lambda_0+\Lambda_1$ and let $w=s_1s_0$. The 
spherical Demazure module $D_{s_1s_0}(\Lambda^{(0)}+\Lambda^{(1)})$ sits inside $D_{s_1s_0}(\Lambda^{(0)})\T D_{s_1s_0}(\Lambda^{(1)})$; it is of dimension $15$ and can be visualized as follows:

\begin{equation}\label{eq:3*5}
\begin{tikzcd}[scale cd=0.5]
	& & \bullet &\bullet &\bullet \\
	& \bullet &\bullet &\bullet &\bullet &\bullet    \\
 \bullet & \bullet &\bullet &\bullet &\bullet &\bullet &\bullet \\
 &
\end{tikzcd}	
\end{equation}
The dots in the picture correspond to the basis vectors, vectors of the same $\fh$-weight are
placed in one column, vectors of the same z-degree are placed in the same rows.
The tensor product $D_{s_1s_0}(\Lambda^{(0)})\T D_{s_1s_0}(\Lambda^{(1)})$ is visualized as 

\[
\begin{tikzcd}[scale cd=0.5]
	& \bullet &\bullet \\
	\bullet &\bullet &\bullet &\bullet\\
	&  
\end{tikzcd}
\qquad \T \qquad
\begin{tikzcd}[scale cd=0.5]
	& \bullet &\bullet \\
	\bullet &\bullet &\bullet &\bullet \\
	& 
\end{tikzcd}	
\]

The Iwahori algebra module $\bfD(0)$ is generated from two extremal vectors (the tensor squares 
of highest and lowest weight vector) and has the following form
  
\begin{equation}\label{eq:3*5Iw}
\begin{tikzcd}[scale cd=0.4]
& &	&\bullet  & _\bullet^\bullet &\bullet &  \\
& & _\bullet^\bullet &_\bullet^\bullet &\bullet & \bullet	& \bullet    \\
 & \bullet & \bullet  &  &  &  &   &\\
& \bullet &  &  &  &   &\\
\bullet &  &  &  &  &  &    
\end{tikzcd}	
\end{equation}
(the doubled points indicate that the corresponding weight-degree space is two-dimensional). In particular, $\dim\bfD(0)=\dim D_{s_1s_0}(\Lambda^{(0)}+\Lambda^{(1)})=15$ and the characters coincide as well (compare \eqref{eq:3*5} and \eqref{eq:3*5Iw}).
We note that this example does not show up in the geometric coherence conjecture, since 
$\Lambda=\Lambda_0+\Lambda_1$ is not a multiple of an affine fundamental weight.  

Now let $\Lambda^{(0)}=\Lambda_0$, $\Lambda^{(1)}=\Lambda_1$ and let $w=s_1s_0s_1$. The 
spherical Demazure module $D_{s_1s_0s_1}(\Lambda^{(0)}+\Lambda^{(1)})$ sits inside $D_{s_1s_0s_1}(\Lambda_0)\T D_{s_1s_0s_1}(\Lambda_1)$; it is of dimension $18$ and can be visualized as follows:

\begin{equation}\label{eq:2*3*3}
\begin{tikzcd}[scale cd=0.5]
	& & \bullet &\bullet \\
	& \bullet &_\bullet^\bullet &_\bullet^\bullet &\bullet  \\
 & \bullet &\bullet &\bullet &\bullet \\
\bullet  & \bullet &\bullet &\bullet &\bullet &\bullet
 &
\end{tikzcd}	
\end{equation}
The tensor product $D_{s_1s_0s_1}(\Lambda_0)\T D_{s_1s_0s_1}(\Lambda_1)$ is visualized as 

\[
\begin{tikzcd}[scale cd=0.4]
	& \bullet  \\
	\bullet &\bullet &\bullet \\
	&  
\end{tikzcd}
\qquad \T \qquad
\begin{tikzcd}[scale cd=0.5]
	& \bullet &\bullet \\
     & \bullet &\bullet \\
	\bullet &\bullet &\bullet &\bullet\\
    &
\end{tikzcd}	
\]

The Iwahori algebra module $\bfD(0)$ is generated from two extremal vectors (the tensor squares 
of highest and lowest weight vector) and has the following form
  
\begin{equation}\label{eq:2*3*3Iw}
\begin{tikzcd}[scale cd=0.4]
& &	&\bullet    \\
& &  _\bullet^\bullet &_\bullet^\bullet &_\bullet^\bullet \\
& \bullet &  _\bullet^\bullet  & _\bullet^\bullet & \bullet & \bullet \\
& \bullet  &\bullet \\
& \bullet   \\
\bullet   
\end{tikzcd}	
\end{equation}
One has $\dim\bfD(0)=\dim D_{s_1s_0s_1}(\Lambda^{(0)}+\Lambda^{(1)})=18$ and the characters coincide as well (compare \eqref{eq:2*3*3} and \eqref{eq:2*3*3Iw}).

\subsection{Rank two affine: non-example}
We start with the Demazure module $D_{s_1s_0}(2\Lambda)$, $\Lambda=\Lambda_0+\Lambda_1$ \eqref{eq:3*5}, but now embed $D_{s_1s_0}(2\Lambda)$ into the tensor product  
$D_{s_1s_0}(\Lambda_0)\T D_{s_1s_0}(\Lambda_0+2\Lambda_1)$ as a Cartan component. This tensor 
product is visualized as
\[
\begin{tikzcd}[scale cd=0.5]
	 &\bullet & \\
	\bullet &\bullet &\bullet \\
	&  
\end{tikzcd}
\qquad \T \qquad
\begin{tikzcd}[scale cd=0.5]
	& \bullet &\bullet &\bullet \\
	\bullet &\bullet &\bullet &\bullet &\bullet \\
	& 
\end{tikzcd}	
\]
Then $\dim\bfD(0)=\dim D_{s_1s_0}(2\Lambda)=15$, but the action of the Iwahori algebra $\fI$ 
generates from the two extremal vector the $14$-dimensional representation which is visualized as follows:

\[
\begin{tikzcd}[scale cd=0.5]
	& &\bullet	& _\bullet^\bullet &\bullet & \bullet  &  \\
	& \bullet & _\bullet^\bullet & \bullet &\bullet & \bullet	& \bullet    \\
	& \bullet &  &  &  &  &  \\
\bullet &  &  &  &   & & 
\end{tikzcd}	
\]

\appendix
\section{Program}\label{sec:app}
In the appendix we describe the computer program which verifies Conjecture 
\ref{conj:main} by comparing the dimension of $\sum_{\sigma\in W} \U(\fI_1) v_{\sigma\barla}$ 
with the dimension of $V_{\sum_b \la^{(b)}}$. In all the performed checks the dimensions do coincide.

\subsection{The setup}
Let $\la=(\la_1\ge\dots\ge\la_n\ge 0)$ be a partition, $V_\la$ the corresponding 
irreducible $\mgl_n$ module with the set of extremal weight vectors  $v_{\sigma\la}\in V_\la$, $\sigma\in S_n$.
The representation $V_\la$ admits the Gelfand-Tsetlin basis labeled by the Gelfand-Tsetlin diagrams (GT patterns) $\underline\la=(\la_{k,i})$ (see \cite{M00}); we denote the
vector corresponding to $\underline\la$ by $v(\underline\la)\in V_\la$. 
The  GT pattern $\underline\la(\sigma)$ corresponding to 
the extremal vector $v_{\sigma\la}$ is described as follows: 
the $k$-th line (the numbers $\la_{k,i}$, $1\le i\le k$) are the numbers 
$\la_{\sigma(1)}, \la_{\sigma(2)},\dots, \la_{\sigma(k)}$ placed in the non-decreasing order.

Our input is a collection of $n$ partitions $\la^{(b)}$, $b\in\bZ/n\bZ$ 
labeled by the vertices of our cyclic quiver, $\la^{(b)}=(\la^{(b)}_1\ge\dots\ge \la^{(b)}_n)$. 
We are working inside the tensor product 
$V=\bigotimes_{b=0}^{n-1} V_{\la^{(b)}}$ of the corresponding irreducible representations.
In what follows we denote by $\underline\lambda^{(b)}$ the GT patterns corresponding to 
the weight  $\la^{(b)}$.
We have a basis of $V$ formed by the tensor products $\otimes_{b=0}^{n-1} v(\underline\lambda^{(b)})$,
where $\underline\lambda^{(b)}$ are arbitrary GT patterns corresponding to $\la^{(b)}$. 

As an output 
the program computes the dimension of a subspace of $V$ generated by the action of $\fI_1$ 
from $n!$ extremal weight vectors
\[
w(\sigma)=\bigotimes_{b=0}^{n-1} v_{\sigma\la^{(b)}},\ \sigma\in S_n. 
\] 
The Lie algebra $\fI_1$ is generated by $n$ operators acting on our tensor product $V$; 
let us denote these operators by $F_b$, $b=0,\dots,n-1$. 
Explicitly, for $b=1,\dots,n-1$ one has 
$F_b=\sum_{\substack{0\le a\le n-1\\ a\ne b}} E^{(a)}_{b+1,b}$ 
and $F_0=\sum_{a=1}^{n-1} E_{1,n}^{(a)}$ (for $x\in\mgl_n$ the operators  $x^{(a)}$
acts as $x$ on the $a$-th factor and as identity on other factors).


The tensor product $V$ is a representation of $\fI_1$ and we are interested in 
the dimension of the subspace  
$S=\mathrm{U}(\fI_1).\mathrm{span}\bigl\{w(\sigma),\ \sigma\in S_n\bigr\}$. 
Our goal is to check that
$S=V_\barla$ (see Section \ref{sec:fin-dim}), which is equivalent to 
$\dim S = \dim V_{\sum_b \la^{(b)}}$.

\subsection{The algorithm}\label{app:alg}
At each stage we have a list of linearly independent vectors from $S$; 
each vector from the list is a weight vector.
In particular, the weight of the extremal vector $v_{\sigma\la}\in V_\la$ 
is equal to $\sigma(\la)$. As a consequence, the weight of the vector 
$w(\sigma)\in \bigotimes_{a=0}^{n-1} V_{\la^{(a)}}$ is equal to $\sum_{a=0}^{n-1} \sigma(\la^{(a)})$.


Given a linearly independent generating set $\mathcal{S}$ (e.g. $\{w(\sigma):\sigma \in S_n\}$) 
the program generates $S$ as follows: First, it defines the set of all vectors to which the 
operators $F_a$ have not yet been applied, setting $V:=\mathcal{S}$. Then, the program proceeds 
by taking out an element $v\in V$ at each iteration and calculating $F_a v$. Non-zero 
results are added to the temporary basis of $S$ only if they are linearly independent of 
the elements of $S$ having the same weight. Such results are also added to $V$ for being used 
in the next iterations. The iteration continue this way until there are no more vectors
to apply the operators on (i.e. $V=\emptyset$).



\subsection{Pseudocode}
We now present the algorithm described in \ref{app:alg} as pseudocode. 
\footnote{The full Python code can be found here: \url{https://github.com/AndreyK19/quiver_subspace}}. 
The function that computes $\dim S$ (and in fact generates a basis for $S$) 
is called \verb|generate_subspace| and can be found in \verb|main.py|. 
Also, for for any vector $u$ let $\mu(u)$ be the weight of $u$. 
The code works as follows:
\begin{algorithm}[H]
\caption{generate\_subspace($(\lambda^{(0)}, \ldots, \lambda^{(n-1)}),\mathcal{S}$)}
\begin{algorithmic}
\State $\mathcal{B} = \emptyset$ \Comment{Basis of $S$}
\State $V = \mathcal{S}$ \Comment{Vectors remained to apply the operators on}
\While{$V \neq \emptyset$}
\State{Pick some $v\in V$}
\State{$V \gets V\setminus\{v\}$}
\For{$a\in \{0,,...n-1\}$}
    \State{$u \gets F_a v$}
    \If{$u = 0$} 
        \State{continue}
    \Else
        \If{$a = 0$}
            \State{$\mu \gets \mu(v) - \delta_n  + \delta_1$}
        \Else
            \State{$\mu \gets \mu(v) - \delta_a  + \delta_{a+1}$}
        \EndIf
        \If{$\{w\in \mathcal{B}: \mu(w)=\mu \} = \emptyset$}
            \State{$\mathcal{B} \gets \mathcal{B} \cup \{u\}$}
            \State{$V \gets V \cup \{u\}$}
        \Else
            \If{$\{w\in \mathcal{B}: \mu(w)=\mu \} \cup \{u\} \text{ is linearly independent}$}
                \State{$\mathcal{B} \gets \mathcal{B} \cup \{u\}$}
                \State{$V \gets V \cup \{u\}$}
            \Else
                \State{continue}
            \EndIf
        \EndIf
    \EndIf
\EndFor
\EndWhile
\State \Return $|\mathcal{B}|$
\end{algorithmic}
\end{algorithm}


\begin{thebibliography}{99}


	
\bibitem[AB24]{AB24}
P.~N.~Achar, A.~Bourque,
\emph{On global Schubert varieties for the general linear group},
arXiv:2412.01962.

\bibitem[AR]{AR}
P.~Achar, S.~Riche, Central sheaves on affine flag varieties, 
https://lmbp.uca.fr/~riche
	
\bibitem[BNvdV20]{BNvdV20}
D.~Bar-Natan, R.~van der Veen, \emph{An unexpected cyclic symmetry of $I\mathfrak{u}_n$}, arXiv:2002.00697, 2020, 1--9.
	
\bibitem[BR24]{BR24}
M.~Bulois, N.~Ressayre, \emph{On the automorphisms of the Drinfel’d double
of a Borel Lie subalgebra}, Journal of Algebra, vol. 647 (2024), pp. 515--532.

\bibitem[F-M25]{F-M25}
F.~Fauquant-Millet, \emph{Symmetric semi-invariants for some Inonu-Wigner contractions I},
Transformation Groups (2025), https://doi.org/10.1007/s00031-024-09897-6.


\bibitem[F11]{F11} 
E.~Feigin,  \emph{Degenerate flag varieties and the median Genocchi numbers},
{Mathematical Research Letters} {18} (2011), no. 6, 1--16.

\bibitem[F12]{F12}
{E.~Feigin}, \emph{${\mathbb G}_a^M$ degeneration of flag varieties},
{Selecta Mathematica} {18}:3 (2012), 513--537.


\bibitem[FFL11]{FFL11}
E.~Feigin, G.~Fourier, P.~Littelmann,
PBW filtration and bases for irreducible modules in type $A_n$.
\emph{Transform. Groups} \textbf{16} (2011), no.~1, 71--89.


\bibitem[FLP23]{FLP23}
E. Feigin, M. Lanini, A. P\"utz, \emph{Totally nonnegative Grassmannians, Grassmann necklaces and quiver Grassmannians}, Canad. J. Math. {\bf 75} (2023), no. 4, 1076--1109. 

\bibitem[FLP24]{FLP24} E.~Feigin,  M.~Lanini, A.~P\"utz, \emph{Generalized juggling patterns, quiver Grassmannians and affine flag varieties}, 
Math. Z. 308, 53 (2024).

\bibitem[FoL06]{FoL06}
G.~Fourier. P.~Littelmann, \emph{Tensor product structure of affine Demazure modules and 
limit constructions}, Nagoya Math. J. (182), pp. 171 -- 198, 2006.

\bibitem[Ga01]{Ga01}
D. Gaitsgory, \emph{Construction of central elements in the affine Hecke algebra via nearby cycles}, Invent. Math. {\bf 144} (2001), no. 2, 253--280.

\bibitem[Go01]{Go01}
U.~G\"ortz, \emph{On the flatness of models of certain Shimura varieties of PEL-type}, Math. Ann. {\bf 321} (2001), 689--727.

\bibitem[Go10]{Go10}
U.~G\"ortz, \emph{Affine Springer fibers and affine Deligne-Lusztig varieties},
Affine flag manifolds and
principal bundles, 1--50, Trends Math., Birkh\"auser, Springer Basel AG, Basel, 2010.


\bibitem[HaN02]{HaN02}
T.~J.~Haines, B.~C.~Ng\^{o},\emph{Nearby cycles for local models of some Shimura varieties}, 
Compositio Math. 133 (2002), pp. 583--642.

\bibitem[He13]{He13} X.~He, \emph{Normality and Cohen–Macaulayness of local models of Shimura Varieties}, Duke Mathematical Journal 162 (2013), no. 13, 2509--2523.

\bibitem[HY24]{HY24}
J.~Hong, H.~Yu, \emph{A refinement of the coherence conjecture of Pappas and Rapoport},
arXiv:2412.15062.

\bibitem[J85]{J85}
A.~Joseph, \emph{On the Demazure character formula}, Annales Scientifique de l'E.N.S.,
1985, 389--419.

\bibitem[Ka21]{Ka21}
S.~Kato, \emph{Frobenius splitting of Schubert varieties of semi-infinite flag manifolds},
Forum of Mathematics, Pi 9 (e5), 1--56.

\bibitem[Kn08]{Kn08} A.~Knutson, \emph{The cyclic Bruhat decomposition of ${\rm Gr}_k(\bC^n)$ from the affine Bruhat decomposition of $AFlag_k^\circ$}, talk at Bert Kostant's 80th birthday conference (2008),  http://pi.math.cornell.edu/~allenk/positroid.pdf 


\bibitem[KR00]{KR00}
R.~Kottwitz, M.~Rapoport, \emph{Minuscule alcoves for $GL_n$ and $GSp_{2n}$}, 
Manuscripta Math. 102 (2000), 403--428.

\bibitem[Kac90]{Kac90}
V. Kac, \emph{Infinite-dimensional Lie algebras}. Third edition. Cambridge University
Press, Cambridge, 1990.

\bibitem[KZ07]{KZ07}
A.~Knutson, P.~Zinn-Justin, \emph{A scheme related to the Brauer loop model}, 
Adv. Math. 214 (2007) 40--77.


\bibitem[Kum88]{Kum88}
S.~Kumar, \emph{Proof of the Parthasarathy-Ranga Rao-Varadarajan conjecture},
Invent. Math. 93.1 (1988), pp. 117--130.

\bibitem[Kum89]{Kum89}
S.~Kumar, \emph{A refinement of the PRV conjecture}, 
Inventiones Mathematicae, 97, (1989), pp. 305--311,

\bibitem[Kum02]{Kum02}
S.~Kumar, Kac–Moody Groups, Their Flag Varieties and Representation Theory. Progress in Mathematics, vol. 204 (Birkh\"auser, Boston, 2002)

\bibitem[KRV21]{KRV21}
M.S.~Kushwaha, K.N.~Raghavan, S.~Viswanath, \emph{A study of Kostant--Kumar modules
via Littelmann paths}, Advances in Mathematics 381 (2021), p. 107614.

\bibitem[KRV24]{KRV24}
M.S.~Kushwaha, K.N.~Raghavan, S.~Viswanath, 
\emph{Crystals for Kostant--Kumar modules of $\widehat\msl_2$}, arXiv:2409.09328. 

\bibitem[M88]{M88}
O.~Mathieu, \emph{Construction du groupe de Kac--Moody et applications}, C.R. Acad.
Sci. Paris, t. 306, 1988, 227--330.

\bibitem[M00]{M00}
A.Molev, \emph{Gelfand-Tsetlin bases for classical Lie algebras}, 
On Gelfand-Tsetlin Bases for Representations of Classical Lie Algebras. In: Krob, D., Mikhalev, A.A., Mikhalev, A.V. (eds) Formal Power Series and Algebraic Combinatorics. Springer, Berlin, Heidelberg, 2000.


\bibitem[Pa18]{Pa18}
G.~Pappas, \emph{Arithmetic models for Shimura varieties}, Proceedings of the ICM -- Rio 2018. Vol. II.
Invited lectures, 377--398, World Sci. Publ., Hackensack, NJ, 2018.

\bibitem[PR08]{PR08}
G.~Pappas, M.~Rapoport,
\emph{Twisted loop groups and their affine flag varieties}, Adv. Math. 219 (2008), no. 1, 118--198.

\bibitem[PRS13]{PRS13}
G.~Pappas, M.~Rapoport, B.~Smithling, \emph{Local models of Shimura varieties},
I. Geometry and combinatorics, Handbook of moduli. Vol. III, Adv. Lect. Math.
(ALM), vol. 26, Int. Press, Somerville, MA, 2013, pp. 135--217.

\bibitem[PY12]{PY12}
D.I.~Panyushev, O.S.~Yakimova, \emph{A remarkable contraction of semisimple Lie algebras,} Ann. Inst. Fourier (Grenoble) 62 (2012) 2053--2068.


\bibitem[Zhou19]{Zhou19}
Q.~Zhou, \emph{Convex polytopes for the central degeneration of the affine Grassmannian}, Adv. Math. {\bf 348} (2019), 541--582.


\bibitem[Zhu14]{Zhu14}
X.~Zhu, \emph{On the coherence conjecture of Pappas and Rapoport}, Ann. of Math. (2) {\bf 180} (2014), no. 1, 1--85.


\bibitem[Zhu17]{Zhu17}X.~Zhu, \emph{An introduction to affine Grassmannians and the geometric Satake equivalence}, Geometry of moduli spaces and representation theory, 59--154, IAS/Park City Math. Ser., 24, Amer. Math. Soc., Providence, RI, 2017.









\end{thebibliography}
\end{document}